\documentclass{amsproc}
\usepackage{euscript,graphicx}
\usepackage{cases}
\usepackage{mathrsfs}
\usepackage{bbm}
\usepackage{amssymb}
\usepackage{txfonts}
\usepackage{amsfonts,amsmath,amsxtra,mathdots,mathabx}
\usepackage{color}
\usepackage{hyperref}
\usepackage{tikz}

\allowdisplaybreaks

\DeclareFontFamily{U}{matha}{\hyphenchar\font45}
\DeclareFontShape{U}{matha}{m}{n}{
	<5> <6> <7> <8> <9> <10> gen * matha
	<10.95> matha10 <12> <14.4> <17.28> <20.74> <24.88> matha12
}{}
\DeclareSymbolFont{matha}{U}{matha}{m}{n}

\DeclareMathSymbol{\Lt}{3}{matha}{"CE}
\DeclareMathSymbol{\Gt}{3}{matha}{"CF}

\DeclareSymbolFont{mathc}{OML}{txmi}{m}{it}
\DeclareMathSymbol{\varvv}{\mathord}{mathc}{118}
\DeclareMathSymbol{\varnu}{\mathord}{mathc}{"17}

\DeclareSymbolFont{mathd}{OML}{ztmcm}{m}{it}
\DeclareMathSymbol{\varalpha}{\mathord}{mathd}{11}
\DeclareMathSymbol{\varlambda}{\mathord}{mathd}{21}
\def\valpha{\text{\scalebox{0.86}{$\varalpha$}}}

\DeclareMathSymbol{\depsilon}{\mathord}{mathd}{15}
\def\vepsilon{\text{\scalebox{0.88}{$\depsilon$}}}

\DeclareMathSymbol{\varchi}{\mathord}{mathd}{31}

\def\tw{\textit{w}}
\def\nd{\mathrm{d}}
\def\oo{\mathrm{o}}

\def\CalJ{\text{\usefont{U}{dutchcal}{m}{n}J}\hskip 0.5pt}
\def\CalL{\text{\usefont{U}{dutchcal}{m}{n}L}\hskip 0.5pt}

\newcommand{\GL}{{\mathrm {GL}}}

\newcommand{\sstyle}{\scriptstyle}

\newcommand{\ra}{\rightarrow}

\def\-{^{-1}}

\def\sasymp{\text{ \small $\asymp$ }}
\def\mod{\mathrm{mod}\, }

\def\sumx{\sideset{}{^\star}\sum}
\def\sumxx{\sideset{}{^{\star\hskip -1pt \star}}\sum}

\def\Var{\text{\it Var\,}}

\def\lp {\left (}
\def\rp {\right )}

\def\Voronoi{Vorono\"{i}   }

\renewcommand{\Im}{{\mathrm{Im} }}
\renewcommand{\Re}{{\mathrm{Re} }}

\def\shskip{\hskip 1pt}

\newcommand{\delete}[1]{}

\theoremstyle{plain}

\newtheorem{thm}{Theorem}[section] \newtheorem{cor}[thm]{Corollary}
\newtheorem{lem}[thm]{Lemma}

\newtheorem {rem}[thm]{Remark}

\numberwithin{equation}{section}

\begin{document}

	\title{Hybrid Weyl-type bound for $p$-power twisted  $\mathrm{GL} (2)$ $L$-functions}
	\author[Z. Gao, S. Luo, and Z. Qi]{Zhengxiao Gao, Shu Luo, and Zhi Qi}
	\address{School of Mathematical Sciences, Zhejiang University, Hangzhou, 310027, China}
	\email{zxgao@zju.edu.cn}
	
	\address{School of Mathematics, Shandong University, Jinan, 250100, China}
	\email{shu.luo@mail.sdu.edu.cn}

	\address{School of Mathematical Sciences, Zhejiang University, Hangzhou, 310027, China}
	\email{zhi.qi@zju.edu.cn}

	\begin{abstract}
	Let $g$ be a fixed holomorphic cusp form of arbitrary level and nebentypus. Let  $\chiup$ be a primitive character of prime-power modulus $q = p^{\gamma}$. 	In this paper, we prove the following hybrid Weyl-type subconvexity bound
		\begin{align*}
			L (1/2 + it, g \otimes \chiup) \Lt_{g, p, \vepsilon} (  (1+|t|) q )^{1/3+ \vepsilon}
		\end{align*}
	for any $\vepsilon > 0$. 
	\end{abstract}

	\subjclass[2010]{11F66}
	\keywords{$L$-functions, subconvexity, Bessel delta method}
	\thanks{The third author was supported by National Key R\&D Program of China No. 2022YFA1005300 and National Natural Science Foundation of China No. 12071420.}
	\maketitle

	\section{Introduction}
	
\subsection*{Backgrounds}

It is a central problem in analytic number theory to bound a certain family of $L$-functions $L(s, \pi)$ on the critical line $\Re (s)= {1}/{2}$. The subconvexity problem is to improve (usually in a sub-family) the trivial convexity bound:
\begin{align*}
	L({1}/{2}+it,\pi)\Lt_{\vepsilon} C(\pi, t)^{1/4+\vepsilon},
\end{align*}
where $C( \pi, t)$ is the so-called analytic conductor.
See \cite{Iwaniec-Sarnak-GAFA} and \cite[\S 5]{IK}.

Two classical results in the $\mathrm{GL}_1$ setting are the Weyl bound for the Riemann zeta function $\zeta (s)$ and the Burgess bound for the Dirichlet $L$-function $L(s, \chiup)$ ($\chiup$ is a Dirichlet character of modulus $q$):
\begin{align}\label{0eq: Weyl}
	\zeta({1}/{2}+it)\Lt_{\vepsilon}(1+|t|)^{1/6+\vepsilon},
\end{align}
  and
\begin{align}\label{0eq: Burgess}
	L({1}/{2}+it,\chiup )\Lt_{t, \vepsilon} q^{3/16+\vepsilon}.
\end{align}
See \cite{Weyl-1916,Weyl-1921,Littlewood,Landau} and \cite[Theorem 5.5]{Titchmarsh} for \eqref{0eq: Weyl} and  \cite{Burgess-1,Burgess-2} for \eqref{0eq: Burgess}. It is a consensus that the Weyl exponent $1/6$ in \eqref{0eq: Weyl} and the Burgess exponent $3/16$ in \eqref{0eq: Burgess} are two natural barriers in the subconvexity problem.

The first instances of hybrid subconvexity bounds for Dirichlet  $L({1}/{2}+it,\chiup )$ with both $t$ and $q$ varying were given by Heath-Brown \cite{Heath-Brown,Heath-Brown2}, and in his second paper \cite{Heath-Brown2}, the hybrid Burgess-type bound was achieved:
\begin{align}\label{0eq: hybrid Burgess}
	L({1}/{2}+it,\chiup )\Lt_{ \vepsilon} (  (1+ |t|) q)^{3/16+\vepsilon}.
\end{align}
Heath-Brown's idea was to generalize and combine the methods of van der Corput and   Burgess.

In the ground-breaking work of Conrey and Iwaniec \cite{CI-Cubic}, by considering the cubic moment of $\GL_2 \times \GL_1$ $L$-central values $L (1/2, g \otimes \chiup)$ over a $\GL_2$-spectral family ($g$ is either Maass or Eisenstein), among other results, the Weyl-type bound 	was proven in the $	q$-aspect:
	 \begin{align}\label{0eq: Conrey-Iwaniec}
	 	L({1}/{2}+it,\chiup )\Lt_{t, \vepsilon} q^{1/6+\vepsilon},
	 \end{align}
under the assumption that $\chiup$ is quadratic and $q$ is square-free.
Later, with the same assumption on $\chiup$, Young \cite{Young-Cubic} proved the hybrid  Weyl-type bound
	\begin{align}\label{0eq: hybrid Conrey-Iwaniec}
		L({1}/{2}+it,\chiup )\Lt_{ \vepsilon} (  (1+|t|) q)^{1/6+\vepsilon}.
	\end{align}
Recently, Petrow and Young  \cite{PY-2020,PY-2019}  showed successfully that the Weyl-type bound \eqref{0eq: hybrid Conrey-Iwaniec} holds for every $\chiup$ unconditionally.

Thanks to the developments of the van der Corput method and the Bombieri--Iwaniec method (we refer the readers to the treatises \cite{GK-vdCorput,Huxley}), many sub-Weyl exponents for $\zeta (1/2+it)$ have been obtained  over the past century. Now Bourgain \cite{Bourgain-2017} has the best exponent $13/84 \approx 0.1547$. On the other hand, Mili\'cevi\'c \cite{Mil-Sub-Weyl} has introduced a $p$-adic analogue of  the van der Corput method only recently, and it enables him to  achieve a sub-Weyl exponent $\approx 0.1645$ for the Dirichlet family $L (1/2, \chiup)$ to $p$-power moduli $q = p^{\gamma}$ (for $\gamma$ large).  However, these are currently the only two instances  for which sub-Weyl subconvexity is known.

In the $\GL_2$ or $\GL_2 \times \GL_1$ setting, the subconvexity problem has been investigated intensively since 1980's.

The first result is the celebrated Weyl-type bound of Good \cite{Good}:
\begin{align}\label{0eq: GL(2)-Weyl}
	L({1}/{2}+it, g )\Lt_{g,   \vepsilon}   (1+|t|)^{1/3+\vepsilon}
\end{align}
for a fixed holomorphic cusp form $g$ of full level. Good used the spectral theory of automorphic functions.

The bound in \eqref{0eq: GL(2)-Weyl} was recovered by Jutila in his treatise on $\GL_2$ exponential sums \cite{Jut87} (see also \cite[\S \S 10, 20]{Huxley}), using Farey dissection, Vorono\"i summation, and van der Corput theory. Jutila's method became quite influential. It was developed in \cite{Meu87,Jutila97,BMN} to extend \eqref{0eq: GL(2)-Weyl} to the cases where $g$ is Maass or holomorphic of arbitrary level and nebentypus. Furthermore, Blomer and Mili\'cevi\'c \cite{Blomer-Mil-p-adic} introduced the $p$-adic analogue of Jutila's method, especially the $p$-adic counterparts to Farey  dissection and van der Corput theory, and they proved the Weyl-type bound in the $q$-aspect when $q$ is a $p$-power. More precisely, for a holomorphic or Maass form $g$ of full level and a primitive character of modulus $q = p^{\gamma}$ ($p$ odd), Theorem 2 in \cite{Blomer-Mil-p-adic} reads
\begin{align}\label{0eq: Blomer-Milicevic}
	L (1/2 + it, g \otimes \chiup) \Lt_{g, \vepsilon} (1+|t|)^{5/2} p^{7/6} q^{1/3+\vepsilon}.
\end{align}
Further, this was generalized to the case of arbitrary $g$ in \cite{Assing}.

The purpose of this paper is to improve \eqref{0eq: Blomer-Milicevic} and prove the hybrid Weyl-type bound in both the $t$- and $q$-aspects. Before stating our theorem, however, we continue to discuss the history of some related $\GL_2 \times \GL_1$ results and methods.

For the case of arbitrary $\chiup$, the first subconvexity bound in the $q$-aspect is due to Duke, Friedlander, and Iwaniec \cite{DFI-1} (see also the note added at the end of \cite{DFI-2}):
\begin{align}\label{0eq: DFI}
	L (1/2 + it, g \otimes \chiup) \Lt_{g, \vepsilon} (1+|t|)^{2}   q^{5/11+\vepsilon}
\end{align}
for holomorphic $g$ of full level. For this, they introduced a new $\delta$-symbol method (usually called the DFI $\delta$-method nowadays) and the amplification technique. Bykovski\u{\i}  \cite{Bykovskii} used his trace formula for certain mean values involving $L(s, g \otimes \chiup)$ to improve the right-hand side of \eqref{0eq: DFI} into $ (1+|t|)^{1/2+\vepsilon} q^{3/8+\vepsilon}$ for holomorphic $g$ of arbitrary level and trivial nebentypus. Blomer and Harcos \cite{BH-Hybrid} pushed Bykovski\u{\i}'s method to its limit and generalized his Burgess-type bound (in the $q$-aspect) to the general case that $g$ may be Maass  and be allowed to have any nebentypus.

The first hybrid subconvexity bound was also obtained in Blomer--Harcos \cite{BH-Hybrid}:
\begin{align}\label{0eq: BH-Hybrid}
	L (1/2 + it, g \otimes \chiup) \Lt_{g, \vepsilon} ((1+|t|)    q)^{19/40+\vepsilon} ,
\end{align}
for any newform $g$ (holomorphic or Maass) of trivial nebentypus; actually their bound also has an explicit dependence on the level of $g$.  Munshi \cite{Munshi-Circle-I} used the DFI $\delta$-method to improve the exponent $19/40$ in \eqref{0eq: BH-Hybrid} into $4/9$ for general $g$ of arbitrary nebentypus but level co-prime to $q$. Following the seminal work of Michel and Venkatesh \cite{Michel-Venkatesh-GL2} in which the $\GL_2$ subconvexity problem is settled (in all aspects),  H. Wu \cite{WuHan-GL2} obtained (over number fields) the Burgess-like exponent $3/8 + \theta /4$, where $\theta$ is  any exponent towards the Ramanujan--Petersson conjecture. Further, by the method of double Dirichlet series in \cite{Hofstein-Husle-MDS}, C. I. Kuan \cite{Kuan-Hybrid} proved a Weyl--Burgess-like hybrid bound:
\begin{align}\label{0eq: Kuan-Hybrid}
	L (1/2 + it, g \otimes \chiup) \Lt_{g, \vepsilon} (1+|t|)^{1/(3-2\theta) + \vepsilon}    q^{3/8+\theta/4+\vepsilon} ,
\end{align}
provided that $g$ is holomorphic and has trivial nebentypus. Recently, in the case that $q$ is prime ($g$ is still arbitrary), via the Bessel $\delta$-method, Y. Fan and Q. Sun \cite{Fan-Sun} obtained the genuine Weyl--Burgess-type hybrid bound:
\begin{align}\label{0eq: Fan-Sun-Hybrid}
	L (1/2 + it, g \otimes \chiup) \Lt_{g, \vepsilon} (1+|t|)^{1/3 + \vepsilon}    q^{3/8 +\vepsilon} .
\end{align}

In the past decade, a new approach via variants of the $\delta$-symbol or circle method has been developed in a series of works by Munshi \cite{Munshi-Circle-I,Munshi-Circle-II,Munshi-Circle-III,Munshi-Circle-IV} to tackle the subconvexity problem for $\GL_3$. It turns out that his ideas are also successful in the $\GL_2$ setting. For example, they are used to recover
\begin{itemize}
	\item [(1)] the Burgess-type bound  of Bykovski\u{\i}, Blomer, and Harcos for $L (1/2,  g \otimes  \chiup)$ in \cite{Munshi17.5,AHLS} (for $q$ prime),
	\item[(2)] the Weyl-type bound of Good for $L(1/2+it, g )$ in \cite{ARSS,Aggarwal-Singh2017,Aggarwal,AHLQ-Bessel-delta},
	\item[(3)] the Weyl-type bound of Blomer and Mili\'cevi\'c for $L (1/2+it, g \otimes \chiup)$ with $\chiup$ of $p$-power moduli in \cite{MS-2019},
\end{itemize}
and, as mentioned earlier,  to improve
\begin{itemize}
	\item [(4)]  the hybrid bound  of Blomer and Harcos for $L(1/2+it, g \otimes \chiup )$ in \cite{Munshi-Circle-I,Fan-Sun}.
\end{itemize}
In (3) the Weyl-type bound of Munshi and Singh \cite[Theorem 1.2]{MS-2019} reads: 
\begin{align}\label{0eq: Munshi-Singh}
	L (1/2 + it, g \otimes \chiup) \Lt_{g, t, \vepsilon}  p^{(\gamma - \lfloor \gamma / 3 \rfloor)/ 2 + \vepsilon \gamma},
\end{align}
for $g$   of square-free level and $\chiup$   of $p$-power modulus $p^{\gamma}$. 

Finally, we remark that hybrid Weyl-type   bounds were achieved  for certain self-dual (twisted) $\GL_2$ $L$-functions in \cite{Young-Cubic,PY-2020,PY-2019}, but  the self-dual assumption is not required here. Moreover, in similar spirits, other Weyl-type bounds for $\GL_2$ (Rankin--Selberg)  $L$-functions were obtained in \cite{JM-Uniform,LLY,BJN-triple,Nelson-Eisenstein,BFW-totally real,WX}.

 \subsection*{Main results}

 Let $g  \in S^{\star}_k (M, \xiup)$ be a holomorphic cusp newform of level $M$, weight $k$, nebentypus character $\xiup$, with the Fourier expansion
 $$g (z) = \sum_{n=1}^{\infty} \lambdaup_g (n) n^{(k-1)/2} e (n z), \quad e(z) = e^{2 \pi i z} ,$$
 for  $\Im  (z) > 0$. Assume that $g$ is Hecke-normalized so that $ \lambdaup_g (1) = 1$.  Let $\chiup$ be a primitive character of $p$-power modulus $p^{\gamma}$.  Recall that the twisted $L$-function $L(s, g \otimes \chiup)$ is defined by
 $$L(s, g \otimes \chiup) = \sum_{n=1}^{\infty} \frac {\lambdaup_g (n) \chiup(n)} {n^{s}} ,$$
 for $\mathrm{Re}(s) > 1$; it extends to an entire function by analytic continuation.

In this paper, we prove the following hybrid bound for $L(1/2 + it, g \otimes \chiup)$ with Weyl strength in both the $t$- and $q$-aspects.

\begin{thm}\label{main-theorem-Weyl}
	Let $g \in S^{\star}_k (M, \xiup)$ and $\chiup \hskip 1pt (\mod p^{\gamma})$ be as above. Suppose   $(p, 2 M) = 1$.  Then
	\begin{equation}\label{0eq: main-Weyl}
		L\left(1/2+it,  g \otimes \chiup \right)\Lt_{g, \vepsilon}  (1+|t|)^{1/3+\vepsilon} p^{ (\gamma - \lfloor \gamma /3 \rfloor )/2 + \vepsilon \gamma },
	\end{equation}
	with the implied constant depending only on $g$ and $\vepsilon$.
\end{thm}

Theorem \ref{main-theorem-Weyl} will be deduced in \S \ref{sec: hybrid-Weyl} from the following estimate for twisted $\GL_2$ exponential sums with $ \phi (x) = - \log x / 2\pi $ and $T = t$.

\begin{thm}\label{main-theorem2}
Let the setting be as in Theorem {\rm\ref{main-theorem-Weyl}}. 	Let $N, T, \varDelta \geqslant 1$. 
	Let $V (x) \in C_c^{\infty} [1, 2] $. Assume that its total variation $ \text{\it Var\,} (V) \Lt 1$ and that $V^{(j)} (x) \Lt_{j} \varDelta^j$  for $j \geqslant 0$. For  $ \phi (x) \in C^{\infty} (1/2, 5/2) $ satisfying   $ |\phi'' (x)| \Gt 1 $ and  $  \phi^{(j)} (x) \Lt_j 1 $ for $j \geqslant 1$, define
	\begin{align}\label{0eq: defn of f(x)}
		f (x) = T \phi (x/ N) .
	\end{align}
Define
\begin{align}
S_{f, \chiup} (N) =	\sum_{n=1}^\infty \lambdaup_g(n) \chiup (n) e ( f(n) ) V\left(\frac{n}{N}\right) .
\end{align} Then
	\begin{equation}\label{1eq: main bound}
	\begin{split}
		S_{f, \chiup} (N) \Lt_{g, \phi, \vepsilon}    (1 + \varDelta   / T) T^{1/3} p^{ (\gamma - \lfloor \gamma /3 \rfloor )/2  }  N^{  1 / 2  +\vepsilon}    + \frac { (1+\varDelta / T)^{1/2} N^{1+\vepsilon}} {  T^{1/6}  p^{  \lfloor \gamma /3 \rfloor /2 } }    ,
	\end{split}
	\end{equation}
	with the implied constant depending only on $g$,  $\phi$ and $\vepsilon$.
\end{thm}

The Rankin--Selberg bound in \eqref{2eq: Ramanujan} implies that $S_{f, \chiup} (N) = O_{g} (N)$, so  \eqref{1eq: main bound} yields a non-trivial result in the range
\begin{align}
   (1+\varDelta / T)^3 p^2  N^{\vepsilon}    <   T  p^{\gamma  } < \frac {N^{3/2-\vepsilon}}   {(1+\varDelta / T)^3 p} .
\end{align}
So  in our arguments we may assume the wider but simpler range:
\begin{align}\label{1eq: range of N}
	(T+\varDelta) p^{\gamma+1} < N^{3/2} .
\end{align}

By choosing $V (x)$ to be a suitable weight function with sharp cut-offs,  the following corollary will be proven in \S \ref{sec: proof of Cor}.

\begin{cor}\label{cor}
Define
\begin{align}
	S^{\sharp}_{f, \chiup} (N) = \sum_{n\leqslant N}  \lambdaup_g(n) \chiup (n) e ( f(n) ).
\end{align}
Then for $T^5<p^{3\lfloor\gamma/3\rfloor}$, we have
	\begin{equation}\label{0eq: cor1}
	S^{\sharp}_{f, \chiup} (N) 	 \Lt_{g, \phi, \vepsilon}   T^{1/3}p^{\lp \gamma-\lfloor\gamma/3\rfloor\rp/2}N^{1/2+\vepsilon} +     \frac{p^{\lp\gamma-\lfloor\gamma/3\rfloor\rp/4}N^{3/4+\vepsilon}}{T^{1/3}}  + \frac {   N^{1+\vepsilon}} {  T^{4/9}  p^{  \lfloor \gamma /3 \rfloor /3 } } ,
	\end{equation}
	and for $T^5\geqslant p^{3\lfloor\gamma/3\rfloor}$, we have
	\begin{equation}\label{0eq: cor2}
		S^{\sharp}_{f, \chiup} (N) 	 \Lt_{g, \phi, \vepsilon}  T^{1/3}p^{\lp \gamma-\lfloor\gamma/3\rfloor\rp/2}N^{1/2+\vepsilon}    + \frac {   N^{1+\vepsilon}} {  T^{1/6}  p^{  \lfloor \gamma /3 \rfloor /2 } } .
	\end{equation}
\end{cor}

Our idea is to incorporate the Bessel $\delta$-method \cite{AHLQ-Bessel-delta} into the approach of  \cite{MS-2019}. The analysis in \cite{AHLQ-Bessel-delta} is simple enough to enable us to obtain a strong Weyl-type bound in the $t$-aspect, and to even go beyond the so-called Weyl barrier   \cite{HMQ}\footnote{A similar result  in the $q$-aspect ($q = p^{\gamma}$) is also obtained in \cite{GM} by applying  the $p$-adic  van der Corput method of Mili\'cevi\'c \cite{Mil-Sub-Weyl}.} (there is an application to the sub-Weyl subconvexity problem, though not enough for a solution). Also, we take the chance to give a corrected version of \cite[Lemma 5.2]{MS-2019} (see \S \ref{sec: char sum E(n)}).

Finally, we comment that, in view of the work \cite{Fan-Sun}, our results are also valid for Maass cusp forms. 

\subsection{Comparison with the results of Blomer and Mili\'cevi\'c} Let $g$ be a holomorphic or Maass cusp form of full level.
It is proven by Blomer and Mili\'cevi\'c \cite[Theorem 1]{Blomer-Mil-p-adic} that
\begin{align}\label{0eq: Bl-Mi's bound, 1}
	\sum_{n=1}^\infty \lambdaup_g(n) \chiup (n) W \lp \frac {n} {N} \rp \Lt_{g,   \vepsilon} Z^{5/2}    p^{7/6+\gamma/3+\vepsilon \gamma} N^{1/2},
\end{align}
for $W (x) \in C_c^{\infty} [1,2]$ with $ W^{(j)} (x) \Lt Z^j $. Now, in the setting of Theorem \ref{main-theorem2}, on choosing $W (x) = e (T\phi (x)) V (x) $ and $ Z = T + \varDelta $, the bound in \eqref{0eq: Bl-Mi's bound, 1} turns into
\begin{align}\label{0eq: Bl-Mi's bound, 2}
	S_{f, \chiup} (N) \Lt_{g, \phi, \vepsilon} (1+\varDelta/T)^{5/2} T^{5/2} p^{7/6+\gamma/3+\vepsilon \gamma} N^{1/2}.
\end{align}
Note that \eqref{1eq: main bound} is better than \eqref{0eq: Bl-Mi's bound, 2} for $ N < (T+\varDelta)^{4-\vepsilon} T^{4/3} p^{5/3 + \gamma  } $.

Another result of Blomer and Mili\'cevi\'c in \cite[(1.5)]{Blomer-Mil-p-adic} reads:
\begin{align}\label{0eq: sharp char sum, BM}
	\sum_{n\leqslant N}  \lambdaup_g(n) \chiup (n) \Lt_{g,  \vepsilon} p^{1/3 + 2\gamma/21 + \vepsilon \gamma } N^{6/7 } .
\end{align}
In Theorem \ref{main-theorem2}, by letting $T = 1$ and absorbing $ e(\phi (x))$ into $V(x)$, the proof of Corollary \ref{cor} yields the following bound:
\begin{align}\label{0eq: sharp char sum}
	 \sum_{n\leqslant N}  \lambdaup_g(n) \chiup (n) \Lt_{g,  \vepsilon}       {p^{\lp\gamma-\lfloor\gamma/3\rfloor\rp/4}N^{3/4+\vepsilon}}   + \frac {   N^{1+\vepsilon}} {    p^{  \lfloor \gamma /3 \rfloor /3 } }.
\end{align}
Note that \eqref{0eq: sharp char sum, BM} beats the trivial bound if $N > p^{7/3 + 2\gamma/3 + \vepsilon \gamma}$, and so does \eqref{0eq: sharp char sum} in a slightly larger range  $ N > p^{2/3 + 2\gamma/3 + \vepsilon \gamma}$ (up to the  factor $N^{\vepsilon}$).
Moreover, in this range \eqref{0eq: sharp char sum} is better than \eqref{0eq: sharp char sum, BM} for $   N < p^{7/9 + 13\gamma/9   } $.

We remark that the Jutila  method has its advantages over the Bessel $\delta$-method when $N$ is  large. This is due to the nature and the limitation of the Bessel $\delta$-method or any $\delta$- or circle method---it is less beneficial to separate the oscillations if either their measure is decreasing or the length of summation is increasing.

\subsection*{Notation} Subsequently, unless otherwise specified, we shall reserve the letter $q$ for  primes, instead of the modulus of the  character $\chiup$. 
The notation $y \sim Y$ stands for $y \in [Y, 2Y]$ (according to the context, $y$ can be integers, primes, or real numbers).

Let $\vepsilon$ or $A$ be an arbitrarily small or large real number, respectively, whose value may differ
from one occurrence to another.


	\section{Preliminaries}
	
	Let $S^{\star}_k (M, \xiup)$ denote the set of (Hecke-normalized) holomorphic newforms of level $M$, weight $k$ and nebentypus $\xiup$. We have necessarily $\xiup (-1) = (-1)^k$.
	Let  $g \in S^{\star}_k (M, \xiup)$.  For its Fourier coefficients $\lambdaup_g (n)$, we have the Rankin--Selberg bound
	\begin{align}\label{2eq: Ramanujan}
		\sum_{n \shskip \leqslant N} |\lambdaup_g (n)|^2 \Lt_{g} N.
	\end{align}
	Moreover,    the  Ramanujan conjecture for holomorphic cusp forms is well-known (\cite{Deligne,Deligne-Serre}):
	\begin{align}\label{2eq: Ramanujan, 0}
		\lambdaup_g (n) \Lt n^{\vepsilon }.
	\end{align}

	
	\subsection{The Vorono\"i summation} 	The following \Voronoi summation formula (\cite[Lemma 2.1]{AHLQ-Bessel-delta}) is a special case of \cite[Theorem A.4]{KMV}. 
	
	\begin{lem}[The \Voronoi Summation Formula]\label{lem: Voronoi}
		Let $g$ be a   holomorphic newform in $S^{\star}_k (M, \xiup)$. Let $a, \widebar{a}, c$ be integers such that $c \geqslant 1$, $(a, c) = 1$, $a \widebar{a} \equiv 1 (\mod c)$ and $(c, M) = 1$. Let $F (x) \in C_c^{\infty} (0, \infty)$. Then there exists a complex number $\etaup_g $ of modulus $1$ {\rm(}the Atkin--Lehner pseudo-eigenvalue of $g${\rm)} such that
		\begin{equation}\label{2eq: Voronoi}
			\begin{split}
				\sum_{n=1}^{\infty}  \lambdaup_g(n)  e\left(\frac{an}{c}\right) F\left( {n} \right) =  \frac {\etaup_g \xiup (-c)} {c \sqrt{M}} \sum_{n=1}^{\infty} \overline { \lambdaup_{  g } (n) } e\left(- \frac{\widebar{a}n}{c}\right) \check{F} \lp \frac{  n }{c^2 M} \rp ,
			\end{split}
		\end{equation}
		where $\check{F} (y)$ is the Hankel transform of $F (x)$ defined by
		\begin{align}\label{2eq: Hankel transform}
			\check{F} (y) = 2\pi i^k\int_0^\infty F(x) J_{k-1}\left( {4\pi\sqrt{xy}} \right) \hskip -1pt \mathrm{d}x,
		\end{align}
	and $J_{k-1} (x)$ is the Bessel function of the first kind.
	\end{lem}

The \Voronoi summation formula in \cite[Theorem A.4]{KMV} is more general, where it is only required that $((c, M), M/(c,M) ) = 1$.

For the Farey dissection in the method of Jutila  \cite{Jut87} or its $p$-adic analogue in Blomer--Mili\'cevi\'c \cite{Blomer-Mil-p-adic},  the fraction $a/c$ has to be arbitrary, so the \Voronoi in \cite{KMV} works only if $M$ is square-free; thus in  \cite{BMN}, they need a more general \Voronoi even without the restriction $((c, M), M/(c,M) ) = 1$.

For the circle method of Kloosterman used in Munshi--Singh \cite{MS-2019} (with conductor lowering), one needs  $c = p^{\valpha} q$ for any $q \leqslant Q$, and hence the square-free condition  on  $M$ is imposed  there.

In comparison, the Bessel $\delta$-method is very flexible on the modulus (the reader may compare it with the Jutila $\delta$-method \cite{Jutila-1,Jutila-2,Jutila-3} used in \cite{Munshi-Circle-I})---we may choose $q \sim Q$ to be prime so that $(M p, q) = 1$ and in our application $c = p^{\valpha} q$ or $p^{\valpha} $ for $\valpha \leqslant 2 \lfloor\gamma/3\rfloor $. Therefore,   Lemma \ref{lem: Voronoi} may be applied under the mild assumption $(p, M) = 1$ (as in Theorem \ref{main-theorem-Weyl}, \ref{main-theorem2}). It is not hard to see that  this technical assumption is removable, as one may resort to  {\rm\cite[Theorem A.{\rm4}]{KMV}} for the case $\valpha \geqslant \mathrm{ord}_p (M) $ and to {\rm\cite[Lemma {\rm2.4}]{BMN}}  for the remaining case $0 < \valpha < \mathrm{ord}_p (M)$.

\subsection{A stationary phase lemma}

Among the three stationary phase lemmas in Appendix A of \cite{AHLQ-Bessel-delta}, only   Lemma A.1 (recorded below) will be used in our paper. However,  we need  their  results   (Lemma 5.1,  5.5) for which Lemmas A.2 and A.3 (the $1$- and $2$-dimensional second derivative tests\footnote{Note that one only needs the weight functions to have bounded variation in the second derivative tests.}) are required. 

\begin{lem}\label{lem: staionary phase}
	Let $\tw \in C_c^{\infty} [a, b]$ and $f \in C^{\infty} [a, b]$ be  real-valued. Suppose that there
	are   parameters $P, U,   Y, Z,  R > 0$ such that
	\begin{align*}
		f^{(i)} (x) \Lt_{ \, i } Y / P^{i}, \hskip 10pt \tw^{(j)} (x) \Lt_{ \, j } Z / U^{j},
	\end{align*}
	for  $i \geqslant 2$ and $j \geqslant 0$, and
	\begin{align*}
		| f' (x) | \Gt R.
	\end{align*}
	Then for any $A \geqslant 0$ we have
	\begin{align*}
		\int_a^b e (f(x)) \tw (x) \nd x \Lt_{ A} (b - a) Z \bigg( \frac {Y} {R^2P^2} + \frac 1 {RP} + \frac 1 {RU} \bigg)^A .
	\end{align*}
\end{lem}



 \subsection{The Bessel $\delta$-method}  Fix a (non-negative valued) bump function $U \in C_c^{\infty} (0, \infty) $ with support in $[1,2]$. For $a, b > 0$ and $X > 1$, consider the Bessel integral
 \begin{align}\label{2eq: defn I(a,b;X)}
 	I_{k} (a, b; X) =  \int_0^\infty U\left( {x}/{X}\right) e ( {2  a \sqrt{  x}}  ) J_{k-1}  ( {4\pi b \sqrt{ x}}  ) \mathrm{d}x.
 \end{align}
Subsequently, the implied constants will always depend on $k$ and $U$, and we shall suppress $k$, $U$ from the subscripts of $O$, $\Lt$ or $\Gt$ for simplicity.

Let $ \widetilde{U} $ denote the Mellin transform of $U$. For $a^2X > 1$, \cite[Lemma 3.1]{AHLQ-Bessel-delta} states that
	 \begin{equation}\label{2eq: asymptotic of I(a,a)}
	 		I_k (a,a;X) = \frac {(1+i)i^{k-1} \widetilde{U}(3/4) X } {4 \pi  ( a^2 X)^{1/4} } + O \left( \frac {X} {(a^2  X)^{3/4 } } \right).
	 \end{equation}
For our problem, the error term in \eqref{2eq: asymptotic of I(a,a)} is undesirable, so we shall replace it by the following lemma.
\begin{lem}\label{lem: I(a,a)}
	 We may write
	\begin{align}\label{2eq: I(a, a) = I (x)}
		I_k (a,a;X) = \frac {i^{k}X} {2\pi  } I_{k}  \big(  {a \sqrt{X}} \big),
	\end{align}
	so that $I_k (x) \in C^{\infty} (0, \infty)$ has bounds
	\begin{align}\label{2eq: bounds for 1/I(x)}
		x^j	\frac {\mathrm{d}^j } {\mathrm{d} x^j}   \frac 1 {I_k (x)} \Lt_j   {\sqrt{x}},
	\end{align}
for any   $  x \Gt  1$. 
\end{lem}

\begin{rem}
	Alternatively,  a cumbersome asymptotic formula of $I_k (a,a;X) $ is obtained by contour shift in Proposition {\rm1.1} of {\rm\cite{Fan-Sun}}.
\end{rem}

\begin{proof}
	Recall from  \cite[\S 3.2]{AHLQ-Bessel-delta} that
	\begin{align*}
		I_k (a,a;X) = \frac{X}{2\pi i}\int_{(\sigma)}\widetilde{U}(s)\frac{2 i^{k-1}}{\sqrt{\pi} ( {-8\pi i a \hskip -1pt \sqrt{  X}}  )^{2-2s}} \frac{ \Gamma(k-2s+1)\Gamma(2s-3/2)} {\Gamma(k+2s-2)}\mathrm{d}s,
	\end{align*}
	for  $ {3}/{4}<  \sigma  < ( k+1)/2 $, so the expression \eqref{2eq: I(a, a) = I (x)} is clear, with
	\begin{align}
		I_k (x) = \frac 1 {2\pi i} \int_{(\sigma)}\widetilde{U}(s) \frac{4 \sqrt{\pi}}{ i ( {-8\pi i x}  )^{2-2s}} \frac{ \Gamma(k-2s+1)\Gamma(2s-3/2)} {\Gamma(k+2s-2)}\mathrm{d}s.
	\end{align}
As in  \cite[\S 3.2]{AHLQ-Bessel-delta}, by shifting the contour of integration to $\Re(s) = 0$ say, and collecting the residues at $s = 3/4$ and $1/4$, we obtain
\begin{align*}
	I_k (x) =   \frac {   \widetilde{U}(3/4)    }  { (1+i) \sqrt{x} } + O  \bigg(  \frac {1}   {\sqrt{x} x  } \bigg),
\end{align*}
 which yields exactly the asymptotic in \eqref{2eq: asymptotic of I(a,a)}. From this, it is clear that
 $$|I_k(x) | \Gt 1 /\sqrt{x} $$
 as long as $x \Gt 1$ is large in terms of $k$ and $U$. Similar to the above, we find that
\begin{align*}
x^j	I_k^{(j)} (x) =   \frac {   \widetilde{U}(3/4)  (2j-1)!!  }  { (1+i) (-2)^j \sqrt{x} } + O_{  j} \bigg(  \frac {1}   {\sqrt{x} x  } \bigg),
\end{align*}
and hence
$$x^j I_k^{(j)} (x) \Lt_j  1/\sqrt{x}, $$
for $x \Gt 1$. Then the bounds for $1 / I_k (x)$ and its derivatives follow immediately.
\end{proof}

Moreover, we record here \cite[Lemma 3.2]{AHLQ-Bessel-delta}.
\begin{lem}\label{pre-key-lemma, 2}
	Suppose that $b^2 X > 1$. Then $  I_k (a, b; X) = O_A (X^{-A})$ for any $A \geqslant 0$ if $ |   a -  b | \sqrt X > X^{\vepsilon}  $. 
\end{lem}

Finally, we refine the Bessel $\delta$-identity in \cite[Lemma 3.3]{AHLQ-Bessel-delta} in the following form.

\begin{lem}\label{lem: Bessel-delta}
	Let $h > 0 $ be integer and $ N, X > 1$ be such that  $X\Gt h^2 / N $ and $ X^{1- \vepsilon } >N$. Let $r, n$ be integers in the dyadic interval $ [N, 2N]$. For any $A \geqslant 0$, we have
	\begin{equation}\label{2eq: Bessel-delta}
		\begin{split}
			\delta({r = n}) = \frac {2\pi r^{1/4}} { i^k  h^{1/2} X^{3/4}  }  \cdot \frac{1} {h}  \sum_{ a (\mod h ) } e \lp \frac {a (n-r) } {h } \rp \cdot I_k \lp \frac {\hskip -1pt \sqrt r} {h},\frac {\hskip -1pt \sqrt n} {h };X \rp  V_k \lp\frac{ \sqrt {rX}} {h} \rp &\\
			  + O_{ A}  (X^{-A}  ), &
		\end{split}
	\end{equation}
	where $V_k (x) = 1 / \hskip -1pt \sqrt{x} \hskip 0.5pt I_k (x)$ has bounds $x^j V_k^{(j)} (x) \Lt_j 1$ for any $x \Gt 1$,  $\delta(r=n)$ is the Kronecker $\delta$ that detects $r=n$, and the implied constants above depend only on $k$, $U$, and possibly $A$ or $j$.
\end{lem}

\begin{proof}
	  Lemma~\ref{pre-key-lemma, 2} implies that $I_k \lp \hskip -1pt \sqrt r/ h , \hskip -1pt \sqrt n/  h ;X \rp$ is negligibly small unless $|r-n| \leqslant X^{ \vepsilon} h \sqrt {N/X}$, while  the exponential sum in \eqref{2eq: Bessel-delta} gives us $r \equiv n\ (\mod h)$. Consequently, the Kronecker $\delta (r=n)$ follows immediately from $X^{ \vepsilon} h \sqrt {N/X} < h$ as assumed. Finally, the proof is completed in view of the expression of $I_k (a, a; X)$ in \eqref{2eq: I(a, a) = I (x)} and the bounds for $1/I_k (x)$ in \eqref{2eq: bounds for 1/I(x)} in Lemma \ref{lem: I(a,a)}.
\end{proof}

\begin{rem}
	A merit of the Bessel $\delta$-method is the flexibility of modulus $h$---it is chosen to be a {\rm(}large{\rm)} prime $h = p$ in {\rm\cite{AHLQ-Bessel-delta}} or a product of {\rm(}large{\rm)} primes $h = p q$ in {\rm\cite{Fan-Sun}} {\rm(}$q$ is the modulus of their $\chiup${\rm)}. In our case, we shall choose $h = p^{\hskip 0.5pt \beta} q$ for  $\beta = 2\lfloor\gamma/3\rfloor$ and prime $q$ large {\rm(}$p^{\gamma}$ is  the modulus of our $\chiup${\rm)}.
\end{rem}

\section{ Application of Bessel $\delta$-identity and \Voronoi summation}

Let $S (N) = S_{f, \chiup} (N)$ denote the exponential sum in Theorem \ref{main-theorem2}:
\begin{align}
S (N) =	\sum_{n=1}^\infty \lambdaup_g(n) \chiup (n) e ( f(n) ) V\left(\frac{n}{N}\right) .
\end{align}
Firstly, we write
\begin{align*}
	S(N) =    \sum_{r=1}^\infty \chiup (r) e ( f(r) ) V\left(\frac{r}{N}\right) \sum_{n=1}^\infty \lambdaup_g(n) \delta (r=n).
\end{align*}
Let $q > \max \{p, M\}$ be prime. Let $  \beta         < \gamma$. As it will turn out to be our final choice, we set
\begin{align}\label{2eq: nv = 2gamma/3}
	 \beta        =2 \lfloor  {\gamma}/{3}  \rfloor.
\end{align}
 By applying the $\delta$-method identity \eqref{2eq: Bessel-delta} in Lemma \ref{lem: Bessel-delta} with $h = p^{\hskip 0.5pt \beta        } q$ and reducing $a / h$ to its  lowest term, we have
\begin{align}\label{3eq: S(N) = Sch}
	S(N) =\sum_{ \valpha        =0}^ \beta        \lp {S}^{\star} (p^ \valpha         q,p^ \beta         q; N, X)+ {S}^{\star} (p^ \valpha          ,p^ \beta         q; N, X)\rp+O  (X^{-A}  ),
\end{align}
with
\begin{equation}\label{3eq: Sstar(N,X)}
	\begin{split}
		{S}^{\star }(c, h; N, X)  =  \frac {2\pi i^k M^{1/2} N^{1/4}  } {\etaup_g h^{3/2} X^{3/4} } & \sum_{r=1}^\infty \chiup(r) e ( f(r) ) V_{\scriptscriptstyle \natural}     \hskip -1pt \left(\frac{r}{N}\right)  \hskip -2pt \underset{a(\mod c)}{\sumx}e \lp -\frac { ar } {c} \rp  \\
		\cdot &\sum_{n=1}^\infty \lambdaup_g(n) e \lp \frac {an } {c}  \rp I_k \lp \frac {\hskip -1pt \sqrt r} {h},\frac {\hskip -1pt \sqrt n} {h};X \rp ,
	\end{split}
\end{equation}
where $ V_{\scriptscriptstyle \natural} (x) = \etaup_g \xiup (-1) M^{-1/2} \cdot x^{1/4} V(x) V_k (\hskip -1pt \sqrt{NXx} / h) $ (recall that $\xiup (-1) = (-1)^k$). As usual, the  superscript $\star$ in the $a$-sum means $(a,c)=1$. It is clear from the bounds for $V (x)$ and $V_k  (x)$ in Theorem \ref{main-theorem2} and Lemma \ref{lem: Bessel-delta} that the  new $V_{\scriptscriptstyle \natural} (x)$ has the same nature as $V (x)$; namely, $V_{_{\scriptscriptstyle \natural}} (x)$ is supported in $[1,2]$, satisfying $ \text{\it Var\,} (V_{\scriptscriptstyle \natural}) \Lt 1$ and $V_{\scriptscriptstyle \natural}^{(j)} (x) \Lt_{j} \varDelta^j$.

For the moment, let
\begin{equation*}
	h=p^ \beta         q; \quad c=p^ \valpha         q, \, p^ \valpha         .
\end{equation*}
Recall from \eqref{2eq: defn I(a,b;X)} that
\begin{align*}
	I_k \lp \frac {\hskip -1pt \sqrt r} {h},\frac {\hskip -1pt \sqrt n} {h};X \rp =  \int_0^\infty U\left( {x}/{X}\right) e \lp \frac {2   \sqrt{ r x}} {h} \rp J_{k-1} \lp \frac {4\pi  \sqrt{ n x}} {h}  \rp  \mathrm{d}x.
\end{align*}
Our assumptions $(p, M) = 1$ and $q > M$ ensure that $(c, M) = 1$.  By applying the \Voronoi summation in \eqref{2eq: Voronoi} in the {reversed} direction,  we infer that
\begin{equation}\label{3eq: Sch after Voronoi}
	\begin{split}
		{S}^{\star}(c, h; N, X)   =\frac {\xiup(c)h^{1/2}N^{1/4} } {c X^{3/4} }&\sum_{r=1}^\infty \chiup(r)  e (f(r)) V_{\scriptscriptstyle \natural}    \hskip -1.5pt \left(\frac{r}{N}\right)    \hskip -1pt\\
		\cdot &\sum_{n=1}^\infty  \overline{\lambdaup_{  g } (n)} S(r,n;c)  e \hskip -1pt \lp \frac {2 \sqrt {n r}} {  \sqrt M c } \rp \hskip -1pt U \hskip -1pt \left(\frac {h^2n} { MX c^2}\right) \hskip -1pt ,
	\end{split}	
\end{equation}
where $S  (r,n;c)$ is the Kloosterman sum
\begin{equation*}
	S  (r,n;c) = \sumx_{ a (\mod c) } e \bigg( \frac{ ar + \widebar{a} n}{c}\bigg).
\end{equation*}


Moreover, we introduce  an average of \eqref{3eq: S(N) = Sch} over primes $q$ in the dyadic segment $[Q, 2Q]$ ($q \sim Q$)  for  parameter $Q > \max\{p, M\}$; say there are   $Q^{\star} \asymp Q / \log Q$ many such primes. Then
\begin{align}\label{3eq: S(N) = Sch(Q)}
	S(N) =\sum_{ \valpha        =0}^ \beta        \lp {S}_{ { \valpha        } \hskip 0.5pt  \beta         }^{0} (N, X, Q)+ {S}_{ { \valpha        } \hskip 0.5pt  \beta         }^{1} (N, X, Q)\rp+O  (X^{-A}  ),
\end{align}
where
\begin{align}\label{3eq: S0 and S1}
	{S}_{ { \valpha        } \hskip 0.5pt  \beta         }^{\tau} (N, X, Q) = \frac 1 {Q^{\star}} \sum_{q\sim Q} {S}^{\star} (p^ \valpha         q^{\tau}, p^ \beta         q; N, X) .
\end{align}

Finally, we consider here the conditions for the parameters  in
Lemma \ref{lem: Bessel-delta}. It will be convenient to introduce $H$ and $K$  in place of $Q$ and $X$ such that
	\begin{align}\label{4eq: XN=P2K2}
H = p^{\hskip 0.5pt  \beta        } Q, \quad	X = {H^2 K^2}/{N},    \quad  N^{\shskip\vepsilon} <    K < T^{1-\vepsilon}.
\end{align}
The first assumption $  X \Gt h^2 / N$ (for $h = p^{\hskip 0.5pt \beta        } q \sim H$) is justified by $  K > N^{\vepsilon}$. The second assumption $ X^{1- \vepsilon } >N$ amounts to
\begin{align}
	\label{4eq: assumption on K}
	H > {N^{1+\vepsilon}} / { K }.
\end{align}

\section{Application of Poisson summation}

Let $S  (n, c; N) $ denote the $r$-sum in \eqref{3eq: Sch after Voronoi}:
\begin{align}\label{4eq: r-sum}
	S  (n, c; N) = \sum_{r=1}^\infty \chiup(r) S(r,n;c) e (f(r))   e \hskip -1pt \lp \frac {2 \sqrt {n r}} {  \sqrt M c } \rp V_{\scriptscriptstyle \natural}    \hskip -1.5pt \left(\frac{r}{N}\right)  .
\end{align}
Recall that $\chiup (r)$ is of modulus $p^{\gamma}$ and $f (r) = T \phi (r/N) $.  For $c = p^{ \valpha        } q^{\tau}$, we apply the Poisson summation of modulus $[c, p^{\gamma}] = p^{\gamma} q^{\tau}$ to transform $S  (n, c; N)$ into
\begin{align}\label{4eq: r-sum after Poisson}
	S  (n, c; N) =   { N} \sum_{r = -\infty}^{\infty}  \mathfrak{C}_{ \chiup}  (n, r, c, [c, p^{\gamma}])  \CalJ ( n, r, c, [c, p^{\gamma}]; N  )
\end{align}
where
\begin{equation}\label{4eq: character sum C}
	\mathfrak{C}_{\chiup} (n, r, c, d) = \frac 1 {d}	\sum_{b (\mod d)} \chiup (b) S(b, n; c) e \hskip -1pt \lp \frac {r b } { d} \rp,
\end{equation}
\begin{equation}\label{4eq: def J}
	\begin{split}
		\CalJ (y,r,c, d; N)=\int_{0 }^{\infty} V_{\scriptscriptstyle \natural}   (x) e \hskip -1pt  \left( T \phi (x ) + \frac{2 \sqrt{Nx y}}{\sqrt M c} -\frac{rNx}{d}\right)   \mathrm{d}x,
	\end{split}
\end{equation}
for $h^2 y / c^2 \sim M X$.

\subsection{Evaluation of the character sum $\mathfrak{C}_{  \chiup}  (n, r, c, d) $} We start by simplifying the character sum   $\mathfrak{C}_{  \chiup}  (n, r, c, d) $   defined as in \eqref{4eq: character sum C}.

First, we  consider the case $(c, d) = (p^{ \valpha        } q,  p^{\gamma} q)$. To this end, we open the Kloosterman sum $S(b, n; c)$  in \eqref{4eq: character sum C}, obtaining
\begin{align*}
\frac 1 {p^\gamma q} \hskip 2pt	\sumx_{ a (\mod  p^ \valpha         q) }  e \hskip -1pt \lp  \frac {\widebar{a} n } {p^ \valpha         q} \rp \sum_{b(\mod   p^\gamma q)} \chiup (b)   e \hskip -1pt \lp \frac {b (r+ap^{\gamma -  \valpha        }) } { p^\gamma q } \rp .
\end{align*}
Then we use reciprocity
$$\frac 1 {p^{\gamma} q} \equiv \frac {\widebar{q}} {p^{\gamma}} + \frac {\widebar{p}^{\gamma}} {q} (\mod 1),$$
as $(p, q) = 1$, to write the inner $b$-sum as
	\begin{align*}
		 \sum_{b_1 (\mod  p^\gamma )} \chiup (b_1) e \hskip -1pt  \lp   \frac { b_1 (r+ap^{\gamma- \valpha        } ) \widebar{q}   } {p^\gamma}  \rp \sum_{b_2 (\mod  q )}   e \hskip -1pt \lp  \frac { b_2 (r+ap^{\gamma- \valpha        } ) \widebar{p}^\gamma } {q}  \rp.
	\end{align*}
Further, the $b_1$-sum equals $ \widebar{\chiup} (r+a p^{\gamma -  \valpha        }) {\chiup} (q) \cdot  \tauup ({\chiup})  $, where $\tauup ({\chiup}) $ is the Gauss sum
\begin{align*}
	 \tauup ({\chiup}) = \sum_{a (\mod  p^\gamma )} \chiup (a) e \lp \frac {a} {p^{\gamma}} \rp,
\end{align*}  while the $b_2$-sum yields the congruence $ a \equiv - r \widebar{p}^{\gamma- \valpha        } (\mod q)$.  Thus we arrive at
\begin{align*}
	\frac { \chiup(q)   \tauup ({\chiup}) } {p^{\gamma}}  \mathop{\sumx_{ a (\mod  p^ \valpha         q) }}_{ a \equiv - r \widebar{p}^{\gamma- \valpha        } (\mod q) } \widebar{\chiup} (r+a p^{\gamma -  \valpha        })  e \hskip -1pt  \lp  \frac {\widebar{a} n } {p^ \valpha         q} \rp.
\end{align*}
Note that one must have $(r, q) = 1$, as the sum is empty if otherwise.  This sum may be simplified by reciprocity, and we conclude that
\begin{align}\label{4eq: character sum, 1}
	\mathfrak{C}_{  \chiup} (n, r, p^{ \valpha        }q, p^{\gamma} q) = \frac { \chiup(q)   \tauup ({\chiup}) } {p^{\gamma}} e \bigg( \hskip -2pt - \frac {p^{\gamma} \widebar{p}^{2 \valpha        }  \widebar{r} n } {q} \bigg) \sumx_{ a (\mod  p^ \valpha         ) } \overline{\chiup} (r+ap^{\gamma -  \valpha        }) e \lp  \frac {\overline{a q} n } {p^ \valpha        } \rp.
\end{align}
Similarly, for the degenerate case $(c, d) = (p^{ \valpha        } ,  p^{\gamma} )$, we have
\begin{align}
	 \mathfrak{C}_{  \chiup} (n, r, p^{ \valpha        } , p^{\gamma} ) = \frac {\tauup ({\chiup}) } {p^{\gamma}}  \sumx_{ a (\mod  p^ \valpha         ) } \overline{\chiup} (r+ap^{\gamma -  \valpha        }) e \lp  \frac {\widebar{a } n } {p^ \valpha        } \rp.
\end{align}
Trivially, as $ |\tauup ({\chiup})| = p^{\gamma/2} $, we have
\begin{align}\label{4eq: bound for C, 0}
	\mathfrak{C}_{  \chiup} (n, r, p^{ \valpha        } , p^{\gamma} ) < p^{ \valpha         - \gamma/2 }.
\end{align}

\subsection{Estimates for the integral $ \CalJ(y,r,c,d;N)$}

For the integral $\CalJ(y,r,c,d;N)$ defined in \eqref{4eq: def J}, we have the following lemmas.

\begin{lem} \label{lem: range for r}
	Let $\varDelta \geqslant 1$. Define
	\begin{align}
		\label{4eq: defn of S}
		S =   {T +    \varDelta N^{\vepsilon}   }    .
	\end{align}
	Suppose that  $V_{\scriptscriptstyle \natural} \in C_c^{\infty} [1, 2]$ has bounds  $V_{\scriptscriptstyle \natural}^{(j)} (x) \Lt_j \varDelta^j$. Let $ h^2 y / c^2 \sim M X $.  Then for any $A \geqslant 0$ we have $ \CalJ(y,r, c, d;N) = O_A (N^{-A})$ whenever $ |r   | \Gt S d / N  $.
\end{lem}
This lemma manifests that  one can effectively truncate the sum at $|r  | \sasymp S d/ N$, at the cost of a negligible error.

\begin{proof}

	For  $ h^2 y / c^2 \sim M X $, the derivative of the phase function in \eqref{4eq: def J}  is equal to
	\begin{align*}
		- \frac {N r} {d} + T \phi'(x) + \frac {\sqrt{N y}} {\sqrt{M x} c} = - \frac {N r} {d} + O \bigg( T + \frac {\sqrt{N X}} {H} \bigg).
	\end{align*}
	By \eqref{4eq: XN=P2K2},
	\begin{align*}
		\max \big\{ T, \sqrt{NX}/ H \big\} = \max \{T, K \} = T.
	\end{align*}
	Therefore, the derivative is   dominated by $ - Nr / d$ since $|Nr/d| \Gt  S > T  $, while the $i$-th derivative ($i \geqslant 2$) is bounded by $O_i(T)$. Finally, we apply Lemma \ref{lem: staionary phase} with $Y = T$, $P   = 1$, $U = 1/\varDelta$, and $R =  S$, so that
	$$\frac {Y} {R^2P^2} + \frac 1 {RP} + \frac 1 {RU} = \frac {T} {(T +  \varDelta N^{\vepsilon})^2 } + \frac 1 {T +   \varDelta N^{\vepsilon}} + \frac {\varDelta} {T+ \varDelta N^{\vepsilon}} = O \lp \frac 1 {N^{\vepsilon}} \rp. $$
	
\end{proof}

\delete{\zxg{In fact, we may drop the $N^\vepsilon$ and assume $|r/d|\geqslant R$. By applying \ref{lem: staionary phase} with $a=1$, $b=1$,
		\begin{equation*}
			Y=T,\quad Q=Z=1,\quad U=1/\varDelta,
		\end{equation*}
		and $R$ as in \eqref{lem: range for r}, we have for any $A\geqslant0$
		\begin{equation*}
			\begin{split}
				\CalJ(y,r, c, d;N)&\Lt_A\lp \frac{T}{R^2N^2}+\frac{\varDelta}{RN}\rp^A=\lp\frac{1}{RN}\cdot\frac{T}{T+\varDelta N^\vepsilon}+\frac{\varDelta}{T+\varDelta N^\vepsilon}\rp^A\\
				&\Lt N^{-\vepsilon A}.
			\end{split}
		\end{equation*}
		Thus the error term is negligibly small.
}}

\begin{lem}\label{lem: J(y,r,m,k,N) <}
	Let $ \Var (V_{\scriptscriptstyle \natural}) \Lt 1$. 	For $ h^2 y / c^2 \sim M X$, we have
	\begin{align}\label{4eq: bound for J}
		\CalJ (y,r, c, d;N) \Lt \frac{1}{\sqrt T}.
	\end{align}
\end{lem}

\begin{proof}
	For $y \sim MX$, $h \sim H$, and $X$, $N$, $H$, $K$ as in \eqref{4eq: XN=P2K2},   define
	\begin{equation}\label{4eq: def Jo}
		\begin{split}
			\CalJ_{\mathrm{o}} (y,r,h; N)=\int_{0 }^{\infty} V_{\scriptscriptstyle \natural}   (x) e \hskip -1pt  \left( T \phi (x/N) + \frac{2 \sqrt{Nx y}}{\sqrt M h} -\frac{rNx}{h}\right)   \mathrm{d}x,
		\end{split}
	\end{equation}
	which was originally defined    in \cite[(5.3)]{AHLQ-Bessel-delta}\footnote{Note that $p \sim P$ is used in \cite{AHLQ-Bessel-delta} in place of our $h \sim H$.}. It is proven in \cite[Lemma 5.1]{AHLQ-Bessel-delta}  that
	\begin{align}\label{4eq: bound for Jo}
		\CalJ_{\oo} (y,r,h; N) \Lt \frac 1 {\sqrt{T}} .
	\end{align}
	Evidently, the $\CalJ_{\oo}$-integral is  a special case of the $\CalJ$-integrals.   However,  the change of variables
	\begin{align}\label{4eq: change of variables}
		y_{\oo} = h^2 y / c^2, \quad N_{\oo} = N h^2/ d^2, \quad h_{\oo} = h^2/ d, \quad \text{($H_{\oo} = H h/d$),}
	\end{align}
	may also turn the $\CalJ$-integral into a  $\CalJ_{\oo}$-integral in the sense that
	\begin{align}
		\CalJ (y,r,c, d; N) = \CalJ_{\oo} (y_{\oo},r,h_{\oo}; N_{\oo}).
	\end{align}
Note that  $N_{\oo} > 1$ or $ N > p^{2\gamma -2\beta}$ is ensured here by \eqref{1eq: range of N} and \eqref{2eq: nv = 2gamma/3}. 	Hence \eqref{4eq: bound for J} is a consequence of \eqref{4eq: bound for Jo}.
\end{proof}

Admittedly, Lemma \ref{lem: J(y,r,m,k,N) <} is an easy application of the second derivative test (see \cite[Lemma 5.1, A.2]{AHLQ-Bessel-delta}). However, the true value of the observation above is for a direct deduction of Lemma \ref{lem: the integral I} below from \cite[Lemma 5.5]{AHLQ-Bessel-delta}. In this way, we may avoid a repetition  as in \cite{Fan-Sun} of the three-page proof of \cite[Lemma 5.5]{AHLQ-Bessel-delta}.

\subsection{Estimates for $ {S}_{    { \protect \valpha}   \beta         }^{0} (N, X, Q)  $}

From \eqref{4eq: r-sum after Poisson},  \eqref{4eq: bound for C, 0}, and Lemma \ref{lem: range for r},  \ref{lem: J(y,r,m,k,N) <},   it follows that
\begin{align*}
	S  (n, p^{ \valpha        } ; N) \Lt    { N}   \sum_{|r| \Lt  S p^{\gamma}/N}  \frac {p^{ \valpha        - \gamma/2} } {\sqrt{T}} + N^{-A} \Lt  \frac {S   p^{\gamma/2+ \valpha        }} {\sqrt{T}} .
\end{align*}
By \eqref{3eq: Sch after Voronoi}, \eqref{4eq: r-sum},  and the Rankin--Selberg bound in \eqref{2eq: Ramanujan} (and Cauchy),  we infer that
\begin{align*}
	{S}^{\star}(p^{ \valpha        }, p^{ \hskip 0.5pt \beta        } q; N, X) &  \Lt \frac {\sqrt{p^{\hskip 0.5pt  \beta   } q}  N^{1/4}    } { p^{ \valpha        }  X^{3/4} } \frac {S p^{\gamma/2+ \valpha        }} {\sqrt{ T} } \sum_{n \sim MX  / p^{2 \beta        -2 \valpha        } q^2} |\lambdaup_{  g } (n)|   \\
	& \Lt_{g} \frac { S (N X)^{1/4 }    p^{\gamma /2  - 3  \beta        / 2 + 2  \valpha        } } {   Q^{3/2} \sqrt{T} }\\
	& = \frac {  S \sqrt{  K} p^{\gamma /2  -    \beta          + 2  \valpha        } } { Q \sqrt{T} },
\end{align*}
where the last equality follows from \eqref{4eq: XN=P2K2}. 
Consequently,
\begin{align}\label{4eq: bound for S0}
	\sum_{ \valpha         = 0}^{ \beta        }  {S}_{ { \valpha        }\hskip 0.5pt   \beta         }^{0} (N, X, Q) \Lt \frac {  S \sqrt{  K} p^{\gamma /2  +  \beta         } } { Q \sqrt{T} }.
\end{align}

\subsection{Formulation for $ {S}_{ {\protect \valpha        }    \beta         }^{1} (N, X, Q)  $}
To summarize, in view of \eqref{3eq: Sch after Voronoi}, \eqref{3eq: S0 and S1}, \eqref{4eq: XN=P2K2}, \eqref{4eq: r-sum}--\eqref{4eq: character sum, 1}, and Lemma \ref{lem: range for r}, with simpler choice of notation, we formulate $ {S}_{ { \valpha        } \hskip 0.5pt  \beta         }^{1} (N, X, Q)   $ in the following way.

\begin{lem}\label{lem: S1 after Voronoi-Poisson}
	We have
	\begin{equation}
		\label{5eq: S1}
		\begin{split}
			{S}_{ { \valpha        }\hskip 0.5pt   \beta         }^{1} (N, X, Q) = & \frac {\xiup(p^{ \valpha        }) \tauup (\chiup)  N^{2} } { p^{\gamma +  \beta         +  \valpha        } Q^{\star} (QK)^{3/2}  } \sum_{n }   \overline{\lambdaup_{  g } (n)}   U \hskip -1pt \left(\frac {p^{2 \beta        -2 \valpha        } n} { MX  }\right)    \\
			& \qquad  \cdot  \sum_{q\sim Q} \hskip -1pt \frac {\xiup \hskip 0.5pt \chiup (q)} {\sqrt{q}} \hskip -1pt \mathop{\sum_{(r, \hskip 0.5pt q) = 1}}_{|r| \Lt S p^{\gamma} q/N } \mathfrak{C}_{  \chiup}^{ \valpha        } (n, r, q) \CalJ_{ \beta         \gamma}  ( p^{2 \beta        -2 \valpha        } n, r, q  ) ,
		\end{split}
	\end{equation}
	up to a negligible error term, where
	\begin{align}\label{4eq: char sum C(nrq)}
		\mathfrak{C}_{  \chiup}^{ \valpha        } (n, r, q) =   e \bigg( \hskip -2pt - \frac {p^{\gamma} \widebar{p}^{2 \valpha        }  \widebar{r} n } {q} \bigg) \sumx_{ a (\mod  p^ \valpha         ) } \overline{\chiup} (r+ap^{\gamma -  \valpha        }) e \bigg(   \frac {\overline{a q} n } {p^ \valpha        } \bigg),
	\end{align}
	\begin{align}\label{4eq: integral J(yrq)}
		\CalJ_{ \beta         \gamma} ( y, r, q  ) = \int_{0 }^{\infty} V_{\scriptscriptstyle \natural}   (x) e    \bigg( T \phi (x ) + \frac{2 \sqrt{Nx y}}{\sqrt M p^{\hskip 0.5pt \beta        } q} -\frac{rNx}{p^{\gamma} q}\bigg)   \mathrm{d}x ,
	\end{align}
for $y \sim MX$.
	
\end{lem}

\section{Application of Cauchy inequality and Poisson summation}


Next, we apply the Cauchy inequality to \eqref{5eq: S1} and the Rankin--Selberg bound for the Fourier coefficients $  {\lambdaup_{  g} (n)} $ as in \eqref{2eq: Ramanujan}, getting
\begin{align}\label{5eq: S1 after Cauchy}
	{S}_{ { \valpha        }\hskip 0.5pt   \beta         }^{1} (N, X, Q)  \Lt_g \frac{N^{3/2}}{ p^{\gamma/2 +  \beta         } Q^{\star} \hskip -1pt \sqrt{QK}  } \cdot \sqrt{T^2_{ \valpha          \beta        } (N, X, Q) },
\end{align}
where $T^2_{ \valpha          \beta        } (N, X, Q) $ is defined by
\begin{align}
	 \sum_{n } U \hskip -1pt \bigg(\frac {p^{2 \beta        -2 \valpha        } n} { MX  }\bigg) \Bigg| \sum_{q\sim Q} \hskip -1pt \frac {\xiup \hskip 0.5pt \chiup (q)} {\sqrt{q}} \hskip -1pt \mathop{\sum_{(r, \hskip 0.5pt q) = 1}}_{|r| \Lt S p^{\gamma} q/N } \mathfrak{C}_{  \chiup}^{ \valpha        } (n, r, q) \CalJ_{ \beta         \gamma}  ( p^{2 \beta        -2 \valpha        } n,  r, q  )\Bigg|^2.
\end{align}
Opening the square and switching the order of summations, we obtain
\begin{equation}\label{5eq: afterCauchy-Schwarz}
		T^2_{ \valpha          \beta        } (N, X, Q) =   \underset{q_1,\shskip q_2 \shskip \sim Q}{\sum \sum}  \frac { \xiup \hskip 0.5pt \chiup (q_1 \widebar{q}_2)  } {\sqrt {q_1 q_2} } \hskip -1.5pt   \mathop{ \underset{\sstyle (r_i ,\hskip 0.5pt q_i)=1}{\sum \sum} }_{\sstyle   |r_i| \Lt  S p^{\gamma} q_i/N } T^2_{ \valpha          \beta        } (n,  r_1, r_2, q_1, q_2; N, X ),
\end{equation}
where $T^2_{ \valpha          \beta        } (n, r_1, r_2, q_1, q_2; N, X )$ is the $n$-sum given by
\begin{align}
 \sum_{n }  \mathfrak{C}_{  \chiup}^{ \valpha        } (n, r_1, q_1) \overline{\mathfrak{C}_{  \chiup}^{ \valpha        } (n, r_2, q_2)} \CalJ_{ \beta         \gamma}  ( p^{2 \beta        -2 \valpha        }  n, r_1, q_1  )  \overline{\CalJ_{ \beta         \gamma}  ( p^{2 \beta        -2 \valpha        }  n, r_2, q_2  )} U \hskip -1pt \bigg(\frac {p^{2 \beta        -2 \valpha        } n} { MX  }\bigg) .
\end{align}
We then apply the Poisson summation with modulus $p^{ \valpha        } q_1 q_2$ to transform this $n$-sum into
\begin{align}\label{5eq: n-sum after Poisson}
	\frac {MX} {p^{2 \beta        -2 \valpha        }  } \sum_{n}  \mathfrak{E}_{\chiup}^{ \valpha        } (n; r_{1},r_{2},q_1,q_2)  \CalL_{ \beta        \gamma}  \bigg( \frac {M Xn} { p^{2 \beta        -  \valpha        } q_1 q_2} ;r_{1},r_{2},q_1,q_2  \bigg),
\end{align}
where the character sum $\mathfrak{E}_{\chiup}^{ \valpha        } (n) \hskip -1pt	= \hskip -1pt	 \mathfrak{E}_{\chiup}^{ \valpha        } (n; r_{1},r_{2},q_1,q_2) \hskip -1pt$ is given by
\begin{align}\label{5eq: D character sum}
\mathfrak{E}_{\chiup}^{ \valpha        } (n) \hskip -1pt	= \hskip -1pt	   \frac 1 {p^{ \valpha        } q_1  q_2} \hskip -1pt	 \sum_{b (\mod p^ \valpha         q_1 q_2)} \hskip -1pt	  \mathfrak{C}_{  \chiup}^{ \valpha        } (b, r_1, q_1) \overline{\mathfrak{C}_{  \chiup}^{ \valpha        } (b, r_2, q_2)} e \bigg(\frac{n b} {p^ \valpha         q_1 q_2}\bigg) \hskip -1pt	,
\end{align}
and the integral $ \CalL_{ \beta        \gamma}  (x) \hskip -1pt	= \hskip -1pt	  \CalL_{ \beta        \gamma}  (x;r_{1},r_{2},q_1,q_2 )  $ is given by
\begin{equation}\label{5eq: L integral}
	\begin{split}
		\CalL_{ \beta        \gamma}  (x) \hskip -1pt = \hskip -2pt \int_{0 }^{\infty} \hskip -1pt  U(y)  \CalJ_{ \beta        \gamma} (MXy,r_1,q_1 ) \overline{\CalJ_{ \beta        \gamma} (MXy,r_2,q_2 )}  e (- x y )\mathrm{d}y.
	\end{split}
\end{equation}	

\subsection{Analysis of the character sum $ \mathfrak{E}_{\chiup}^{ \protect\valpha        } (n)$}\label{sec: char sum E(n)} By inserting \eqref{4eq: char sum C(nrq)} into \eqref{5eq: D character sum}, we obtain
\begin{align*}
	 \mathfrak{E}_{\chiup}^{ \valpha        } (n) & =   \frac 1 {p^{ \valpha        } q_1  q_2} \underset{a_1, a_2 (\mod p^ \valpha        )}{\sumx \hskip -2pt \sumx}   \overline{\chiup} (r_1+a_1 p^{\gamma -  \valpha        }) \chiup (r_2+a_2 p^{\gamma -  \valpha        }) \\
	 & \cdot \sum_{b (\mod p^ \valpha         q_1 q_2)} e \bigg(
	  \frac {p^\gamma \widebar{p}^{2 \valpha        } \hskip 1pt \widebar{r}_2 b } {q_2} -  \frac {p^\gamma \widebar{p}^{2 \valpha        }  \widebar{r}_1 b } {q_1}
	 + \frac {\overline{a_1   q_1} b } {p^ \valpha        }
	 - \frac {\overline{a_2 q_2} b } {p^ \valpha        }
	 +\frac{n b} {p^ \valpha         q_1 q_2} \bigg).
\end{align*}
By reciprocity, the exponential $b$-sum yields two congruence relations:
\begin{align}
\label{5eq: congruence 1}	  \widebar{a}_1 & q_2 -\widebar{a}_2q_1  + n \equiv 0 \, (\mod p^ \valpha        ), \\
\label{5eq: congruence 2}	p^{\gamma- \valpha        } ( &  \widebar{r}_2    q_1  -  \widebar{r}_1  q_2) + n \equiv 0 \, (\mod q_1q_2).
\end{align}
Therefore,
\begin{align}\label{5eq: defn of E(n)}
	 \mathfrak{E}_{\chiup}^{ \valpha        } (n) = & \sumx_{a  (\mod p^ \valpha        )}   \overline{\chiup} (r_1+a  p^{\gamma -  \valpha        }) \chiup (r_2+ \overline{ \widebar{a}  q_2+n} \cdot  {q}_1 p^{\gamma -  \valpha        }) ,
\end{align}
if \eqref{5eq: congruence 2} holds, and $ \mathfrak{E}_{\chiup}^{ \valpha        } (n) = 0$ if otherwise.

\begin{lem}\label{lem: bound for C}
 Let $  \valpha         \leqslant 2 \lfloor \gamma/3 \rfloor$. Let $p$ be odd.  Define $\mathfrak{E}_{\chiup}^{ \valpha        } (n)$ as in {\rm\eqref{5eq: defn of E(n)}} with $(q_1 q_2, p) = 1$ {\rm(}here $q_1$ and $q_2$ are not necessarily prime as far as the definition is concerned{\rm)}.

{\rm(1)} If $n \nequiv 0 \, (\mod p^{ \valpha        })$, then $\mathfrak{E}_{\chiup}^{ \valpha        } (n) $ vanishes unless $ r_1 q_1 \equiv r_2 q_2 (\mod p^{  \mathrm{ord}_p(n) })  $, in which case
	\begin{equation}\label{5eq: bound for E(n)}
		\mathfrak{E}_{\chiup}^{ \valpha        } (n) \Lt 	\left\{
		\begin{split}
			&	p^{ \lceil    \valpha         /2 \rceil +   \mathrm{ord}_p(n)    }, & & \text{ if } n \nequiv 0 \, (\mod p^{ \lfloor \valpha  /2 \rfloor      }), \\
			&	p^{ \lceil    \valpha         /2 \rceil +   \mathrm{ord}_p(n) /2   }, & & \text{ if otherwise.}
		\end{split}\right.
	\end{equation}

{\rm(2)} If  $n \equiv 0 \, (\mod p^{ \valpha        })$, then $\mathfrak{E}_{\chiup}^{ \valpha        } (n)$ vanishes unless $ r_1 q_1 \equiv r_2 q_2 (\mod p^{  \valpha -1 })  $, in which case
\begin{equation}\label{5eq: E(0) =}
	 \mathfrak{E}_{\chiup}^{ \valpha        } (n) = \left\{
	 \begin{split}
	 	& \chiup (\widebar{r}_1 r_2  ) \cdot p^{\valpha-1} (p-1), & & \text{ if } r_1 q_1 \equiv r_2 q_2 (\mod p^{     \valpha     }),  \\
	 	& - \chiup ( \widebar{r}_1 r_2   ) \cdot p^{\valpha-1}, & & \text{ if otherwise}.
	 \end{split}
	 \right.
\end{equation}
\end{lem}


\vskip 5pt

\begin{proof}\footnote{Note that there is a careless mistake in Lemma 5.2 of \cite{MS-2019}: The  $ \valpha q \hskip 1pt  \overline{n + q'} $ in their character sum should be $ \valpha q  \hskip 1pt  \overline{ {\valpha} n + q'} =  q \hskip 1pt  \overline{ n + \widebar{\valpha} q'}$. As such, a proof of Lemma \ref{lem: bound for C} is needed here (it is actually more involved).}
	First of all, since $ \mathfrak{E}_{\chiup}^{ \valpha        } (n; r_1, r_2, q_1, q_2) = \chiup ( {q}_1 \widebar{q}_2)  {\mathfrak{E}}_{\chiup}^{ \valpha        } (n; r_1 \widebar{q}_2, r_2  \widebar{q}_1, 1, 1 )$, we may assume $q_1=q_2 = 1$, that is,
	\begin{align}\label{5eq: tilde E(n)}
		 {\mathfrak{E}}_{\chiup}^{ \valpha        } (n) =  \sumx_{a  (\mod p^ \valpha        )}   \overline{\chiup} (r_1+a  p^{\gamma -  \valpha        }) \chiup (r_2+ \overline{ \widebar{a} + n }     p^{\gamma -  \valpha        }) .
	\end{align}
Moreover, if we let  $\delta = \mathrm{ord}_p (n)$, then it can be assumed that $n = p^{\delta}$ if $\delta < \valpha$ (by $   a \ra  \overline{n/p^{\delta}} \cdot  a$) or $n = 0$ if otherwise.

For simplicity, suppose that 
$ \valpha         = 2 \nu$ is even with $0 < \nu \leqslant \gamma/3$; the odd case may be treated in the same way. Note that Lemma \ref{lem: bound for C} is trivial for $ \valpha         = 0, 1$. In our later analysis, we shall always split  $a (\mod p^{2\nu})$ into
\begin{align}\label{5eq: a = a0+a1p}
	a = a_0 + a_1 p^{\nu} , \qquad a_0, a_1 (\mod p^{\nu}), \, (a_0, p) = 1.
\end{align}

The next   lemma will be very helpful. 

 \begin{lem}\label{lem: sum of chi}
Let $ \mu < \nu  \leqslant \gamma/2$. 	
Then for $(\varw, p) = 1$ we have
 	\begin{equation}\label{5eq: sum of chi}
 		\sum_{a_1 (\mod p^{\nu})} \chiup (u + \varv \varw  a_1 p^{\gamma - \nu}) = \left\{ \begin{split}
 			&\chiup (u) \cdot p^{\nu} , & & \text{ if } p^{\nu} | \varv, \\
 			& 0, & & \text{ if otherwise},
 		\end{split}\right.
 	\end{equation}
 \begin{equation}\label{5eq: sumx of chi}
 	\sumx_{a_0 (\mod p^{\nu})} \chiup (u + \varv \varw  a_0 p^{\gamma - \nu}) = \left\{ \begin{split}
 		&\chiup (u) \cdot p^{\nu-1} (p-1), & & \text{ if } p^{\nu} | \varv, \\
 		& - \chiup (u) \cdot   p^{\nu-1}, & & \text{ if } p^{\nu-1} \| \varv, \\
 		& 0, & & \text{ if otherwise},
 	\end{split}\right.
 \end{equation}
and for odd $p $  and  $(\varw  \varw' , p ) = 1$ we have
\begin{equation}\label{5eq: sumx of chi, 3}
	\Bigg| \  \sumx_{a_0 (\mod p^{\nu})} \chiup (u + \varv \varw   a_0 p^{\gamma - \nu} + \varw' a_0^2 p^{\gamma-\nu+\mu})\Bigg| = \left\{  \begin{split}
		&	p^{(\nu+\mu)/ 2}, & & \text{ if } p^{\hskip 0.5pt \mu} \| \varv, \\
			& 0, & & \text{ if otherwise}.
		\end{split} \right.
	\end{equation}

 \end{lem}

To prove Lemma \ref{lem: sum of chi}, one may let $u = \varw   = 1$ and use the fact that $\psiup (x) = \chiup (1+ x p^{\gamma - \nu})$ is a primitive additive character modulo $p^{\nu}$. Note that \eqref{5eq: sum of chi} is just the orthogonal relation, while \eqref{5eq: sumx of chi} is reduced to an evaluation of Ramanujan sums (see \cite[(3.3)]{IK}). For \eqref{5eq: sumx of chi, 3}, we have the Gauss sum
\begin{align*}
	\sumx_{a_0 (\mod p^{\nu})} \psiup (\varv a_0 + \varw' p^{\hskip 0.5pt \mu} a_0^2 ).
\end{align*}   To evaluate this, we first reduce the modulus of $\psiup$ by $(\varv, p^{\hskip 0.5pt \mu})$, and then resort to \cite[Lemma 12.2, 12.3]{IK}. Their $h (y)$ is  linear in our case, so it is easy to see (with $p$ odd) that the sum is non-zero if and only if $ p^{\hskip 0.5pt \mu} \| \varv $, in which case the Gauss sum above has norm $ p^{(\nu+\mu)/ 2} $.

\vskip 5pt

\subsubsection*{Proof of Lemma {\rm\ref{lem: bound for C} (2):} Case $n \equiv 0 \, (\mod p^{2\nu})$}   In this case, it is clear that  $ {\mathfrak{E}}_{\chiup}^{2\nu}(n) =  {\mathfrak{E}}_{\chiup}^{2\nu} (0)$ (see \eqref{5eq: tilde E(n)}) is equal to
\begin{align}\label{5eq: n=0}
	 \sumx_{a (\mod p^{2\nu}) }  \widebar{\chiup} (r_1 + a p^{\gamma - 2\nu}) \chiup (r_2 + {a} p^{\gamma - 2\nu}).
\end{align}
By inserting $a =   a_0 + a_1 p^{\nu} $ (as in \eqref{5eq: a = a0+a1p}), and using the congruence relation
\begin{align}\label{5eq: congruence r1+...}
	\overline{r_1 + a_0 p^{\gamma - 2\nu} + a_1 p^{\gamma - \nu} }\equiv (\overline{r_1 + a_0 p^{\gamma - 2\nu}  })^2   ( r_1 + a_0 p^{\gamma - 2\nu} -  a_1 p^{\gamma - \nu}  )  (\mod p^{\gamma}),
\end{align}
we obtain
\begin{align}\label{5eq: case n=0, after simpl}
{\mathfrak{E}}_{\chiup}^{2\nu} (0) =	\sumx_{a_0 (\mod p^{\nu})} \sum_{a_1 (\mod p^{\nu})} \chiup  (u_0 (a_0) + \varv  \cdot \varw  (a_0) a_1 p^{\gamma - \nu}  ) ,
\end{align}
where
\begin{align*}
	u_0 (a_0) = (\overline{r_1 + a_0 p^{\gamma - 2\nu}   })  (r_2 & + a_0 p^{\gamma - 2\nu}   ) , \quad \varv    = r_1 - r_2, \quad  \varw  (a_0) = (\overline{r_1 + a_0 p^{\gamma - 2\nu}})^2.
\end{align*}
Note that 
$  \gamma \leqslant 2 \gamma-3\nu $ (or $ \nu \leqslant \gamma/3$) is indeed needed here (and used implicitly hereafter). An application of \eqref{5eq: sum of chi} in Lemma \ref{lem: sum of chi} 
implies that ${\mathfrak{E}}_{\chiup}^{2\nu} (0)$ can be non-vanishing only if $ r_1 \equiv r_2 (\mod p^{\nu}) $, in which case we have
\begin{align}
{\mathfrak{E}}_{\chiup}^{2\nu} (0) = p^{\nu} \cdot	\sumx_{a_0 (\mod p^{\nu})} \chiup (u_0 (a_0)).
\end{align}
Further, it follows from $ r_1 \equiv r_2 (\mod p^{\nu}) $ that
\begin{align*}
	u_0 (a_0) & \equiv \widebar{r}_1  (1 - \widebar{r}_1 a_0 p^{\gamma - 2\nu} + \widebar{r}_1^2 a_0^2 p^{2\gamma-4\nu}  ) (r_2 + a_0 p^{\gamma - 2\nu}   )
	\\
	& \equiv \widebar{r}_1  r_2 + \widebar{r}_1^2 (r_1-r_2) a_0 p^{\gamma - 2\nu}  \,(\mod p^{\gamma}) ,
\end{align*}
and hence we arrive at
\begin{align}
{\mathfrak{E}}_{\chiup}^{2\nu} (0) =	p^{\nu} \cdot	\sumx_{a_0 (\mod p^{\nu})} \chiup (u + \varv ' \varw a_0 p^{\gamma - \nu}) ,
\end{align}
with
$$u  = \widebar{r}_1  r_2, \qquad \varv ' = (r_1-r_2)/p^{\nu}, \qquad  \varw  = \widebar{r}_1^2   . $$
Thus \eqref{5eq: E(0) =} is a direct consequence of \eqref{5eq: sumx of chi} in Lemma \ref{lem: sum of chi}.  Note  that ${\mathfrak{E}}_{\chiup}^{2\nu} (0)$ is non-zero if and only if $ (r_1-r_2)/p^{\nu} \equiv 0 \, (\mod p^{\nu-1}) $ or in other words $r_1 \equiv r_2 \, (\mod p^{2\nu-1}) $.

 \vskip 5pt

\subsubsection*{Proof of   Lemma {\rm\ref{lem: bound for C} (1)}}  
For  $n \nequiv 0 \, (\mod p^{  2 \nu        })$, we need to treat the cases $(n, p)=1$ or $ p$ separately.   Recall that  the ${\mathfrak{E}}_{\chiup}^{2\nu} (n) $ under consideration is given by \eqref{5eq: tilde E(n)}.

 \vskip 5pt
	
\subsubsection*{Case $(n, p) = 1$}    Assume $n = 1$.  It follows from the change of variables $a +  1  \ra a  $, $r_1 - p^{\gamma - 2\nu} \ra r_1 $, and $r_2 + p^{\gamma - 2\nu} \ra r_2  $ that
\begin{align}
	{\mathfrak{E}}_{\chiup}^{2\nu} (1) = {\sumxx_{a (\mod p^{2\nu}) }}
	 \widebar{\chiup} (r_1 + a p^{\gamma - 2\nu}) \chiup (r_2 - \widebar{a} p^{\gamma - 2\nu})  ,
\end{align}
where the $\star\hskip -1pt \star$ in the $a$-sum indicates $(a (a-1), p) = 1$.
Now let $a = a_0 + a_1 p^{\nu}$ as in \eqref{5eq: a = a0+a1p}. The congruence in \eqref{5eq: congruence r1+...} together with
\begin{align}
	\overline{a_0 + a_1 p^{\nu}} \equiv \widebar{a}_0 - \widebar{a}_0^2 a_1 p^{\nu} \, (\mod p^{2\nu}),
\end{align}
yields the expression
\begin{align}
	{\mathfrak{E}}_{\chiup}^{2\nu} (1) = {\sumxx_{a_0 (\mod p^{\nu}) }}
	 \sum_{a_1 (\mod p^{\nu}) } \chiup  (u  (a_0) + \varv (a_0)   \varw  (a_0) a_1 p^{\gamma - \nu}  ),
\end{align}
with
\begin{align*}
	u (a_0) = (\overline{r_1 + a_0 p^{\gamma-2\nu}   })  (  r_2 - \widebar {a}_0 p^{\gamma-2\nu}   ) , \qquad   \varv (a_0)   =   r_1 \widebar{a}_0^2 - r_2  .
\end{align*}
It follows from \eqref{5eq: sum of chi} in Lemma \ref{lem: sum of chi} that the character sum is equal to
\begin{align}
{\mathfrak{E}}_{\chiup}^{2\nu} (1) = p^{\nu} \cdot	\mathop{\sumxx_{a_0 (\mod p^{\nu}) } }_{   {a}_0^2 \equiv   r_1 \widebar{r}_2 (\mod p^{\nu}) } \chiup (u (a_0)) ,
\end{align}
and trivially bounded by $2 p^{\nu}$ as the quadratic   equation $ {a}_0^2 \equiv   r_1 \widebar{r}_2 (\mod p^{\nu})$ has at most $2$ solutions for $p$ odd. 

\subsubsection*{Case $(n, p) = p$} Assume $n = p^{\delta}$ with $0 < \delta < 2\nu$. We   need to consider
\begin{align*}
{\mathfrak{E}}_{\chiup}^{2\nu} (p^{\delta}) =	\sumx_{a  (\mod p^{2\nu})}   \overline{\chiup} (r_1+a  p^{\gamma - 2\nu}) \chiup (r_2+ \overline{ \widebar{a} + p^{\delta} } \cdot   p^{\gamma - 2\nu}) .
\end{align*}
By
\begin{align*}
	 \overline{\widebar{a} + p^{\delta}   } = a \cdot \overline{1+  ap^{\delta}   }  \equiv a - a^2 p^{\delta} + a^3 p^{2\delta} - ... \, (\mod p^{2\nu}),
\end{align*}
and $a =   a_0 + a_1 p^{\nu} $  (see \eqref{5eq: a = a0+a1p}), we have
\begin{equation}\label{5eq: congruence a+p}
\begin{split}
	 \overline{\widebar{a} + p^{\delta}} &  \equiv  a_0 + a_1 p^{\nu}  - (a_0^2 + 2 a_0 a_1 p^{\nu}) p^{ \delta } + (a_0^3 + 3 a_0^2 a_1 p^{\nu}) p^{2\delta} - ...  \\
 & \equiv \overline{\widebar{a}_0 + p^{\delta}} + (\overline{1+ a_0 p^{\delta}})^2 a_1 p^{\nu}    \, (\mod p^{2\nu}).
\end{split}
\end{equation}
Similar to the cases above, by \eqref{5eq: congruence r1+...} and \eqref{5eq: congruence a+p}, we have
\begin{align}\label{key}
	{\mathfrak{E}}_{\chiup}^{2\nu} (p^{\delta}) = {\sumx_{a_0 (\mod p^{\nu}) }}  \sum_{a_1 (\mod p^{\nu}) } \chiup  (u_{\delta}  (a_0) + \varv_{\delta} (a_0)   \varw   (a_0) a_1 p^{\gamma - \nu}  ),
\end{align}
with
\begin{align*}
	u_{\delta}  (a_0) = \overline{(r_1 + a_0 p^{\gamma-2\nu})   } & (r_2 + \overline{\widebar{a}_0 + p^{\delta}}  p^{\gamma-2\nu}   ) , \qquad  	\varv_{\delta} (a_0)   = r_1 (\overline{1+ a_0 p^{\delta}})^2  - r_2 .
\end{align*}
Next we apply  \eqref{5eq: sum of chi} in Lemma \ref{lem: sum of chi} in two different cases. Our goal is to prove that the character sum ${\mathfrak{E}}_{\chiup}^{2\nu} (p^{\delta})$ is non-zero only if $ r_1 \equiv r_2 (\mod p^{\delta}) $, and bounded by $  p^{\nu + \delta/2}  $ for $ \delta \geqslant \nu $ or by $  p^{\nu + \delta}  $ for $ \delta < \nu $.

For $ \delta \geqslant \nu $, since $\varv_{\delta} (a_0) $ is reduced to $\varv = r_1 - r_2$ modulo $p^{\nu}$, by  \eqref{5eq: sum of chi} we have
\begin{align}
	{\mathfrak{E}}_{\chiup}^{2\nu} (p^{\delta}) = p^{\nu}\cdot	{\sumx_{a_0 (\mod p^{\nu}) }}  \chiup (u_{\delta} (a_0))
\end{align}
under the condition $r_1 \equiv r_2 (\mod p^{\nu})$. Now
\begin{align*}
	u_{\delta} (a_0) & \equiv \widebar{r}_1^2 (r_1 - a_0 p^{\gamma - 2\nu} + \widebar{r}_1 a_0^2 p^{2\gamma-4\nu}  ) (r_2 + a_0 p^{\gamma - 2\nu}   + a_0^2 p^{\gamma - 2\nu +\delta})
	\\
	& \equiv \widebar{r}_1  r_2 + \widebar{r}_1^2 (r_1-r_2) a_0 p^{\gamma - 2\nu} + \widebar{r}_1  a_0^2 p^{\gamma - 2\nu +\delta} \,(\mod p^{\gamma}) .
\end{align*}
Then we  arrive at
\begin{align}
{\mathfrak{E}}_{\chiup}^{2\nu} (p^{\delta}) =	p^{\nu} \cdot	\sumx_{a_0 (\mod p^{\nu})} \chiup (u + \varv ' \varw  a_0 p^{\gamma - \nu} + \varw'  a_0^2 p^{\gamma -2\gamma +\delta} ) ,
\end{align}
with
$$u  = \widebar{r}_1  r_2, \qquad \varv ' = (r_1-r_2)/p^{\nu}, \qquad  \varw  = \widebar{r}_1^2, \qquad \varw'  = \widebar{r}_1   . $$
Thus  \eqref{5eq: sumx of chi, 3} in Lemma \ref{lem: sum of chi} implies that this sum is  zero unless  $ (r_1-r_2)/p^{\nu} \equiv 0 \, (\mod p^{\delta-\nu}) $ (in other words $r_1 \equiv r_2 \, (\mod p^{\delta}) $) and  bounded by $p^{\nu + \delta/2}$.  

For $ \delta < \nu $, one has $    (1+a_0p^{\delta})^2 r_2 \equiv r_1 (\mod p^{\nu}) $, and this yields $r_1 \equiv r_2 (\mod p^{\delta})$ together with  $ 2 a_0 + a_0^2 p^{\delta}    \equiv (r_1 \widebar{r}_2 -1)/p^{\delta} (\mod p^{\nu-\delta}) $. Consequently,   \eqref{5eq: sum of chi} yields the sum
\begin{align}
{\mathfrak{E}}_{\chiup}^{2\nu} (p^{\delta}) =	p^{\nu} \cdot \mathop{\sumx_{a_0 (\mod p^{\nu})}}_{2 a_0 + a_0^2 p^{\delta}    \equiv (r_1 \widebar{r}_2 -1)/p^{\delta} (\mod p^{\nu-\delta})  } \chiup (u_{\delta} (a_0)).
\end{align}
By Hensel's lemma, the sum has at most $  p^{\delta}$ many terms, so it is bounded by $   p^{\nu+\delta} $.
\end{proof}

\subsection{Bounds for the integral $  \CalL_{ \beta        \gamma}  (x)$}

Recall that the integral  $\CalL_{ \beta        \gamma} (x)$ is defined by {\rm\eqref{4eq: integral J(yrq)}} and {\rm\eqref{5eq: L integral}}.

\begin{lem}\label{lem: the integral I}
	Let $S, T, K, Q > 1$, $N > p^{2\gamma - 2  \beta        }$ be parameters with  $N^{\vepsilon} < K \Lt T \Lt S$ and $ p^{\hskip 0.5pt \beta        } Q > {N^{1+\vepsilon}} / { K } $. 
	Let $q_i \sim Q$ and $|r_i  | \Lt S p^{\gamma} Q/N$ {\rm(}$i = 1, 2${\rm)}. Let $ \Var (V_{\scriptscriptstyle \natural}) \Lt 1$.  Suppose that  $  \phi^{(j)} (\varv) \Lt  1 $ for $j = 2, 3$ and that $ |\phi'' (\varv)| \Gt 1 $ for all $\varv \in (1/2, 5/2)$.
	
	{\rm (1)} We have $\CalL_{ \beta        \gamma} (x) = O (N^{-A})$ if $ |x| \geqslant K  $.
	
	{\rm(2)} Assume that $K^2 / T > N^{\vepsilon}$. For  $ K^2 / T  \Lt |x| < K $, we have
	\begin{align}\label{5eq: bound for I(x)}
		\CalL_{ \beta        \gamma} (x) \Lt  \frac {1}  {\textstyle T \hskip -1pt \sqrt  {|x|}   } .
	\end{align}
	{For $ |x| \Lt K^2 / T $, we have
		\begin{align}\label{5eq: bound for I(x), 2}
			\CalL_{ \beta        \gamma} (x) \Lt \frac  {1}   T.
	\end{align} }
	
	{\rm(3)} Let $q_1 = q_2 = q$.  Then
	\begin{align}\label{5eq: bound for I(0)}
		\CalL_{ \beta        \gamma} (0) \Lt   \min \left\{  \frac 1 T, \frac { S p^{\gamma} Q N^{\vepsilon} } {TKN |r_1-r_2|} \right\}.
	\end{align}
\end{lem}

\begin{proof}
	This lemma is a generalization of \cite[Lemma 5.5]{AHLQ-Bessel-delta}---if we let $ p^{\hskip 0.5pt \beta        } = p^{\gamma} =1$  and $S = T$, then we arrive at their Lemma 5.5\footnote{Note that $p_i \sim P$ in \cite{AHLQ-Bessel-delta} corresponds to our $p^{\hskip 0.5pt \beta        } q_i \sim p^{\hskip 0.5pt \beta        } Q$.}. It is easy to see that their proof for (3) is still valid if several $T$ were replaced by $S$ provided that $S \Gt T$. By the simple observation  in the proof of Lemma \ref{lem: J(y,r,m,k,N) <}, this lemma is actually a consequence of  (the $S$-version of) \cite[Lemma 5.5]{AHLQ-Bessel-delta} by the change of variables $ N_{\oo} = N/ p^{2\gamma -2  \beta        } $, $q_{\oo \hskip 0.5pt i} = p^{2 \beta        - \gamma} q_i $, and $Q_{\oo} =  p^{2 \beta        -\gamma} Q$ (see \eqref{4eq: change of variables}).
\end{proof}

\begin{rem}
	By checking the proof of Lemma {\rm 5.5} in \cite{AHLQ-Bessel-delta}, it seems that the condition  $ p^{\hskip 0.5pt  \beta        } Q > {N^{1+\vepsilon}} / { K } $  arising from the Bessel $\delta$-method {\rm(}see {\rm\eqref{4eq: assumption on K}}{\rm)} is redundant---we keep it here only for the sake of safety.
\end{rem}

	\subsection{Estimates for $ {S}_{ {\protect\valpha}\hskip 0.5pt  \beta }^{1} (N, X, Q)$.}
Now we are ready to  estimate    $ {S}_{ {\valpha}\hskip 0.5pt  \beta }^{1} (N, X, Q)$. 
In view of \eqref{5eq: S1 after Cauchy}--\eqref{5eq: n-sum after Poisson}, \eqref{5eq: congruence 2}, and Lemma~\ref{lem: the integral I} (1), we have
\begin{align}\label{5eq: S<Sdiag+Soff}
	{S}_{ {\valpha}\hskip 0.5pt  \beta }^{1} (N, X, Q)  \Lt_{g} \textstyle \hskip -1pt \sqrt {  {S}_{\mathrm{diag}}^2 (N,   X,   Q)} + \hskip -1pt \sqrt {  {S}_{\mathrm{off}}^2 (N,   X,   Q)} + N^{-A},
\end{align}
with
\begin{equation}\label{S(diag)}
	\begin{split}
		{{S}_{\mathrm{diag}}^2}(N,   X,   Q) = \frac{N^{3 } X }{p^{\gamma + 4\beta -2\valpha}  Q^{\star 2}  Q^2 K }
		\sum_{q \shskip \sim Q}
		&
		\mathop{\mathop{\mathop{\sum \sum}_{(r_1 r_2,\shskip q)=1}}_{\sstyle |r_i| \Lt  S p^{\gamma} q/N}}_{r_1\equiv \, r_2 (\mod q)} \mathfrak{E}_{\chiup}^{\valpha} (0; r_{1},r_{2},q ,q  ) \CalL_{\beta\gamma} ( 0;r_{1},r_{2},q ,q   )  	,
	\end{split}
\end{equation}
\begin{equation}\label{S(off)}
	\begin{split}
		& {S}_{\mathrm{off}}^2   (N,   X,   Q) =    \frac{N^{3}X}{p^{\gamma + 4\beta -2\valpha} {Q^{\star 2}} \hskip -1pt  Q^2 K  } \underset{q_1,\shskip q_2 \shskip \sim Q}{\sum \sum}   \mathop{\mathop{\sum \sum}_{\sstyle (r_i,q_i)=1}}_{|r_i| \Lt  S p^{\gamma} q_i/N }  \\
		\hskip 10pt  & \cdot
	    \mathop{\mathop{\sum}_{0< |n| \Lt N / p^\valpha K}}_{n \equiv p^{\gamma-\valpha} (  \widebar{r}_1  q_2-\widebar{r}_2    q_1)    (\mod q_1q_2)} \hskip -1pt   \mathfrak{E}_{\chiup}^{\valpha} (n;  r_{1},r_{2},q_1,q_2) \CalL_{\beta\gamma} \bigg(\frac{ M X n }{p^{2\beta -\valpha} q_1 q_2 };  r_{1},r_{2},q_1,q_2 \bigg)  ,
\end{split}\end{equation}
in correspondence to the cases where $n=0$ and $n\neq 0$, respectively. Note that for $n = 0$, the congruence relation in  \eqref{5eq: congruence 2} reads $  \widebar{r}_2 q_1 \equiv \widebar{r}_1 q_2 (\mod q_1 q_2)$, and it forces $q_1 = q_2$($=q$) and $r_1 \equiv r_2 (\mod q)$. Moreover, $X = p^{2\beta} Q^2 K^2 / N $ (as in \eqref{4eq: XN=P2K2}) is used here   for the range of the $n$-sum.

For $ {S}_{\mathrm{diag}}^2(N, X, Q)$, we split the sum over $r_1$ and $r_2$ according as $r_1 = r_2$ or not, and apply Lemmas~\ref{lem: bound for C} (2) and~\ref{lem: the integral I} (3). Consequently, ${S}_{\mathrm{diag}}^2(N,X,Q)$ is bounded by
\begin{align*}
 \frac{N^{3 } X }{p^{\gamma + 4\beta -2\valpha}  Q^{\star 2}  Q^2 K }	\sum_{q \shskip \sim Q}
	\Bigg( 	\mathop{\mathop{\sum}_{(r ,\shskip q)=1 }}_{|r | \Lt  S p^{\gamma} q/N}  \frac {p^{\valpha}} T
	 \hskip 1pt
	+  \sum_{\delta = 0, 1}  \mathop{\mathop{\mathop{	\mathop{\sum \sum}_{r_1 \neq r_2  }}_{ (r_1 r_2,\shskip q)=1 }}_{|r_i| \Lt  S p^{\gamma} q/N}}_{r_1\equiv \, r_2 (\mod p^{\valpha-\delta}q)}
 \frac {S p^{\gamma+\valpha-\delta} QN^\vepsilon } {TKN |r_1-r_2|}
	\Bigg)
\end{align*}
and hence
\begin{equation}\label{5eq: estimate for S diag}
	\begin{split}
		{S}_{\mathrm{diag}}^2(N,X,Q) &\Lt \frac{ N^{3 } X }{p^{\gamma+4\beta-2\valpha} Q^{\star }  Q^2 K  } \lp   \frac{S p^{\gamma} Q} {N} \frac {p^{\valpha}} T   +  \frac{S p^{\gamma} Q} {N}  \hskip 1pt \frac {S p^{\gamma } N^\vepsilon } {TKN }  \rp\\
		& \Lt \lp  {  {  K N} p^{ 3\valpha-2\beta} } +   { S } { p^{ \gamma - 2\beta +2\valpha } }  N^{\vepsilon} \rp  \frac{S}{T} \log Q.
	\end{split}
\end{equation}

To deal with $ {S}_{\mathrm{off}}^2(N, X, Q)$, we need to strengthen the condition $Q  > N^{1+\vepsilon} /p^{\hskip 0.5pt \beta} K$ (see \eqref{4eq: assumption on K}) into
\begin{align}\label{5eq: enlarge Q}
	Q >    {N^{1+\vepsilon}} / { K},
\end{align}
so that we would necessarily have $q_1\neq q_2$. Otherwise, if $q_1=q_2=q$, the congruence condition $ n \equiv p^{\gamma-\valpha} (  \widebar{r}_1  q -\widebar{r}_2    q)    (\mod q^2)  $ would imply $q | n$, but this is impossible, in view of the length $  N / p^\valpha K$ of the $n$-sum.

For ease of exposition, we shall only treat the partial sum $  {S}_{\mathrm{off}}^{\star 2}(N, X, Q)$ with the co-prime condition $(n, p) = 1$.
For general $n$, one just needs to use Lemma \ref{lem: bound for C} at   full strength.


Next, we interchange the sum over $n$ and the sums over $r_1$, $r_2$. Note that for fixed $n$, the congruence  $n \equiv p^{\gamma-\valpha} (  \widebar{r}_1  q_2-\widebar{r}_2    q_1)  \, (\mod q_1q_2)  $ splits into  $r_1\equiv \widebar{n}  p^{\gamma-\valpha} q_2  (\mod q_1)$ and $r_2\equiv -\widebar{n}  p^{\gamma-\valpha} q_1  (\mod q_2)$. By Lemma  \ref{lem: bound for C} (1) and \ref{lem: the integral I} (2),  we infer that $  {S}_{\mathrm{off}}^{\star 2}(N, X, Q)$ is bounded by
\begin{align*}
	\frac{N^{3}X}{p^{\gamma + 4\beta -2\valpha} {Q^{\star 2}} \hskip -1pt  Q^2 K  } \underset{q_1,\shskip q_2 \shskip \sim Q}{\sum \sum}
	\Bigg(& \mathop{\sum_{ (n, p) = 1 }}_{ |n| \Lt  N /p^{\valpha}T} \underset{\sstyle  \sstyle |r_i| \Lt  S p^{\gamma} q_i/N  \atop {\sstyle r_1\equiv \widebar{n} p^{\gamma - \valpha} q_2 (\mod q_1)  \atop{\sstyle r_2 \equiv -\widebar{n} p^{\gamma - \valpha} q_1 (\mod q_2)} }}{\sum \sum}   \frac{p^{ \lceil \valpha /2 \rceil }}{T}\\
	&+ \mathop{\sum_{ (n, p) = 1 }}_{N/ p^{\valpha} T  \Lt |n| \Lt  N / p^{\valpha}K} \underset{\sstyle  \sstyle |r_i| \Lt  S p^{\gamma} q_i/N  \atop {\sstyle r_1\equiv \widebar{n} p^{\gamma - \valpha} q_2 (\mod q_1)  \atop{\sstyle r_2 \equiv -\widebar{n} p^{\gamma - \valpha} q_1 (\mod q_2)} }}{\sum \sum} \frac{p^{ \lceil \valpha /2 \rceil } p^{\hskip 0.5pt \beta - \valpha/2}\sqrt{ q_{1}q_{2}}}{T\sqrt{X|n|}} \Bigg).
\end{align*}
We record here the condition in Lemma \ref{lem: the integral I} (2):
\begin{align}\label{5eq: condition on K, 2}
	K > \sqrt{T} N^{\vepsilon}.
\end{align}
When   $N \leqslant S p^{\gamma}$, we have
\begin{equation}\label{5eq. Soff N leq Sp}
	\begin{split}
		{S}_{\mathrm{off}}^{\star 2}(N,   X,   Q) &\Lt \frac{N^{3}X}{p^{\gamma + 4\beta -2\valpha}  \hskip -1pt  Q^2 K  } \cdot p^{ \lceil \valpha /2 \rceil }  \lp \frac{N }{ p^{  \valpha  } T^2 } + \frac{p^{\hskip 0.5pt \beta -   \valpha /2  } Q}{ T \sqrt X} \sqrt{\frac{N }{ p^{  \valpha  } K}} \rp \lp \frac{S p^\gamma }{N }\rp^2 \\
		&= p^{\gamma+  \lceil 3\valpha /2 \rceil  -2\beta  } \lp    { K  N}   + \frac{  NT }{ \sqrt {   K }} \rp \lp \frac S T \rp^2   .
	\end{split}
\end{equation}
When $N > S p^{\gamma}$,   the $(S p^\gamma/N)^2$ in \eqref{5eq. Soff N leq Sp} needs to be replaced by $1$. In other words, we lose $(N/S p^\gamma)^2 $. However, the loss may be reduced to $N/S p^\gamma $ if we rearrange the sums  in the following order:
\begin{align*}
 	\underset{q_{1}\sim Q}{\sum} \hskip -1pt
	\underset{\sstyle(r_{1}, q_{1})=1 \atop{\sstyle |r_1| \Lt  S p^{\gamma} q_1/N}}{\sum} \hskip -5pt
	\Bigg(
	\mathop{\underset{(n, p)=1  } {\sum}}_{ |n|\Lt  N/ p^{\valpha}T} + \mathop{\sum_{(n, p)=1   } }_{N /p^{\valpha}T \Lt |n|\Lt N/p^{\valpha}K} \Bigg) \hskip -5pt
	\underset{\sstyle q_{2}\sim Q \atop{\sstyle q_{2}\equiv r_1 n  \widebar{ p}^{\gamma-\valpha}(\mod q_{1})}}{\sum}
	\underset{\sstyle |r_2| \Lt  S p^{\gamma} q_2/N \atop{\sstyle r_{2}\equiv- \widebar{n}p^{\gamma-\valpha}q_{1}(\mod q_{2})}}{\sum} .
\end{align*}
Thus for $N > S p^{\gamma}$, we have
\begin{equation}\label{5eq. Soff N geq Sp}
	\begin{split}
		{S}_{\mathrm{off}}^{\star 2} (N,   X,   Q)
		&\Lt  	\frac{N^{3}X}{p^{\gamma + 4\beta -2\valpha} {Q^{\star }} \hskip -1pt  Q^2 K  }   \frac{S p^\gamma Q }{N } \cdot p^{ \lceil \valpha /2 \rceil }
		\lp \frac{ N }{p^{  \valpha   } T^{2}}  +  \frac{p^{\hskip 0.5pt \beta -   \valpha /2   } Q }{T\sqrt X}\sqrt{\frac{N }{p^{  \valpha   } K}} \rp\\
		&\Lt \frac {N} {p^{2\beta  - \lceil 3\valpha /2 \rceil   } T}   \lp   {K N}     +  \frac{ N T}{   \sqrt {  K }}  \rp       \frac S T    \log Q.
	\end{split}
\end{equation}
Combining \eqref{5eq. Soff N leq Sp} and \eqref{5eq. Soff N geq Sp}, we have
\begin{equation}\label{5eq: estimate for S off}
	{S}_{\mathrm{off}}^{\star 2} (N,   X,   Q)\Lt p^{\gamma+ \lceil 3\valpha /2 \rceil    -2\beta} \lp   {K N}   +  \frac{ N T}{   \sqrt { K }} \rp  \lp 1 + \frac{  N}{ S p^{\gamma} }\rp   \lp \frac S T \rp^2   \log Q.
\end{equation}

Finally, we conclude from  \eqref{5eq: S<Sdiag+Soff}, \eqref{5eq: estimate for S diag} and \eqref{5eq: estimate for S off} that
\begin{align}\label{5eq: bound for S1}
\sum_{\valpha =0}^{\beta} \hskip -1pt	{S}_{ {\valpha}\hskip 0.5pt  \beta }^{1} (N, X, Q)
\hskip -1pt	\Lt \hskip -2pt
	\bigg(\hskip -2pt    {\sqrt{ Tp^{\gamma}}}
	+ p^{\gamma/2 - \beta/4}   \bigg( \hskip -2pt \sqrt{ K N} + \frac{ \sqrt{NT}  }{ K^{1/4} } \bigg) \hskip -1pt
	\bigg(  \hskip -1pt 1 + \sqrt {\frac { N } {S  p^{\gamma } }} \bigg)  \hskip -1pt \bigg) \frac { S } {T}   N^{\vepsilon}.
\end{align}

\section{Proof of Theorem \ref{main-theorem2}}

Recall that  $\beta = 2 \lfloor \gamma/3 \rfloor$  and $S = T + \varDelta N^{\vepsilon}$ as in \eqref{2eq: nv = 2gamma/3} and \eqref{4eq: defn of S}, so the bound in \eqref{1eq: main bound} that we need to prove is translated into\footnote{Note that the $N^{\vepsilon}$ is removed from $S=  T + \varDelta N^{\vepsilon}$ in \eqref{1eq: main bound} due to our $\vepsilon$-convention.}:
 \begin{align}\label{6eq: final bound}
 	S(N) \Lt \frac {S p^{\gamma/2 - \beta/4} N^{1/2+\vepsilon}} {T^{2/3}}       + \frac { S^{1/2} N^{1+\vepsilon}} {  T^{2/3} p^{\hskip 0.5pt \beta/4} }  .
 \end{align}
Moreover, we may impose the condition (see \eqref{1eq: range of N})
\begin{align}\label{6eq: range of N}
	T^{1/3} p^{\gamma/3} < N^{1/2},
\end{align}
because if otherwise \eqref{6eq: final bound} is worse than the trivial bound $S(N)\Lt N $.

It follows from \eqref{3eq: S(N) = Sch(Q)}, \eqref{4eq: bound for S0} and \eqref{5eq: bound for S1} that
\begin{align*}
	S (N) \Lt \frac {  S \sqrt{  K} p^{\gamma /2  +  \beta         } } { Q \sqrt{T} } + \frac { S p^{\gamma/2} N^{\vepsilon} } {\sqrt{T}} + \frac {S p^{\gamma/2 - \beta/4} N^{\vepsilon}} {T} \bigg( \hskip -2pt \sqrt{ K N} + \frac{ \sqrt{NT}  }{ \sqrt[4]{K}} \bigg) \hskip -1pt
	\bigg(  \hskip -1pt 1 + \sqrt {\frac { N } {S  p^{\gamma } }} \bigg).
 \end{align*}
On choosing $ K = T^{2/3} $ and $Q = N^{1+\vepsilon}/ \sqrt{T}$, we have
\begin{align}\label{6eq: bound after KQ}
	S (N) \Lt \frac {  S  T^{1/3} p^{\gamma /2  +  \beta         }  } { N   } + \frac { S p^{\gamma/2} N^{\vepsilon} } {{T}^{1/2}} + \frac {S p^{\gamma/2 - \beta/4} N^{1/2+\vepsilon}} {T^{2/3}}       + \frac { S^{1/2} N^{1+\vepsilon}} {  T^{2/3}  p^{\hskip 0.5pt \beta/4}}
	   .
\end{align}
The required conditions in \eqref{4eq: XN=P2K2}, \eqref{5eq: enlarge Q}, and \eqref{5eq: condition on K, 2} are well justified for our choice of $K$ and $Q$. Finally, the condition \eqref{6eq: range of N} implies that  the first two terms are dominated by the third term in \eqref{6eq: bound after KQ}, so we arrive at the desired bound \eqref{6eq: final bound}.

\section{Proof of Corollary \ref{cor}}\label{sec: proof of Cor}

It suffices to prove the same estimates in \eqref{0eq: cor1} and \eqref{0eq: cor2} for the sum over dyadic segments,
\begin{equation*}
	S^\flat(N)=\sum_{N \leqslant n\leqslant 2N} \lambdaup_g(n)\chiup(n) e\lp f(n)  \rp.
\end{equation*}
To this end, choose the   $V (x) \in C_c^{\infty} [1, 2]$ in Theorem \ref{main-theorem2} so that $ V (x)\equiv 1 $ on $[1+1/\varDelta, 2-1/\varDelta ]$ and compare $ S^\flat(N) $ with $S (N)$. By   Deligne's bound \eqref{2eq: Ramanujan, 0}, we infer that
\begin{equation*}
	S^\flat(N) =S(N ) + O\lp N^{1+\vepsilon }  /  \varDelta \rp.
\end{equation*}
In order to have cleaner exponents, we consider the case   $\gamma \equiv 0 (\mod 3) $. It follows from \eqref{1eq: main bound} that
\begin{align}\label{7eq: case D<T}
	S^\flat(N) \Lt T^{1/3} p^{  \gamma  /3   }  N^{  1 / 2  +\vepsilon}    + \frac {   N^{1+\vepsilon}} {  T^{1/6}  p^{    \gamma /6 } } + \frac {N^{1+\vepsilon }} {\varDelta}
\end{align}
for $\varDelta \leqslant T$, and
\begin{align} \label{7eq: case T<D}
	S^\flat(N) \Lt \frac {\varDelta  p^{   \gamma /3    }  N^{  1 / 2  +\vepsilon}} {T^{2/3}}    + \frac { \varDelta^{1/2}  N^{1+\vepsilon}} {  T^{2/3}  p^{    \gamma / 6 } } + \frac {N^{1+\vepsilon }} {\varDelta}
\end{align}
for $\varDelta > T$. Note that in the former case, the best choice is clearly $\varDelta = T$, while in the latter case,  the optimal choice of $\varDelta$ is the one that balances the last and the first two terms.

For $T^5 < p^{\gamma}$,
\begin{itemize}
	\item [(1)] if  $N \geqslant T^{4/9}p^{10 \gamma/9}$, we choose
	$\varDelta=T^{4/9} p^{ \gamma/9}$ so that \eqref{7eq: case T<D} yields $$ S^\flat(N) \Lt \frac {N^{1+\vepsilon }}   {T^{4/9} p^{ \gamma/9}} ;  $$
	\item [(2)] if $ T^{8/3} p^{2\gamma/3} \leqslant N < T^{4/9}p^{10 \gamma/9} $, we choose $\varDelta = T^{1/3} N^{1/4} / p^{\gamma/6}$ so that \eqref{7eq: case T<D} yields $$ S^\flat(N) \Lt \frac  {p^{ \gamma/6} N^{3/4+\vepsilon}}   {T^{1/3}}; $$
	\item [(3)] if $ N < T^{8/3} p^{2\gamma/3} $, we choose $\varDelta = T$ and \eqref{7eq: case D<T} yields $$ S^\flat(N) \Lt T^{1/3} p^{  \gamma  /3   }  N^{  1 / 2  +\vepsilon}  .$$
\end{itemize}
Consequently, we obtain \eqref{0eq: cor1} by  combining these bounds.

For $ T^5 \geqslant p^{\gamma}$, we   choose $\varDelta = T$ and \eqref{0eq: cor2} follows directly from \eqref{7eq: case D<T}.

\section{Proof of Theorem \ref{main-theorem-Weyl}}\label{sec: hybrid-Weyl}

Theorem \ref{main-theorem-Weyl} is a standard consequence of Theorem \ref{main-theorem2}. Let  $t > 1$ say.  We have the following approximate functional equation (see \cite[Theorem 5.3, Proposition 5.4]{IK}):
\begin{equation*}
	\begin{split}
		L (1/2 + it, g \otimes \chiup) =& \, \sum_{n=1} ^\infty  \frac{\lambdaup_g(n)\chiup(n)}{n^{1/2+ it}} V_t \bigg( \frac{n}{\sqrt{M} p^\gamma}\bigg) \\
		& +  \epsilon(1/2+it, g\otimes \chiup) \cdot \sum_{n=1} ^\infty \frac{\overline{\lambdaup_g(n) \chiup(n)} }{n^{1/2 -it}} V_{-t}\bigg( \frac{n}{\sqrt{M}p^\gamma} \bigg),
	\end{split}
\end{equation*}
where $ | \epsilon(1/2+it, g\otimes \chiup)| = 1 $, and  $V_{\pm t} (x)$ is a smooth function with $x^jV^{(j)}_{\pm t} (x) \Lt_{j,A}(1+x/ t)^{-A}$ for any $j,A \geqslant 0$. By applying a dyadic partition of unity to the sums above, it follows that for some $ N \leqslant (t p^{\gamma} )^{1+\vepsilon}$ we have
\begin{equation*}
	L(1/2+it,g\otimes\chiup) \Lt_{\vepsilon, A}    \frac{|S(N)|}{ \sqrt N } N^{\vepsilon} + N^{-A},
\end{equation*}
where
\begin{equation*}
	\begin{split}
		S(N)=\sum_{n=1}^{\infty}\lambdaup _g(n) \chiup (n) n^{-it} V \lp \frac{n}{N} \rp,
\end{split}\end{equation*}
and $V (x ) \in C_c^{\infty}[1, 2]$ with  $V^{(j)} (x) \Lt_{j} 1$.
Thus Theorem \ref{main-theorem-Weyl} follows directly from Theorem \ref{main-theorem2} on choosing $\phi(x) = -\log x/2\pi$, $T=t$, and $\varDelta =1$.

\def\cprime{$'$} \def\cprime{$'$}

\end{document}